\documentclass[10pt,reqno,oneside]{article}

\usepackage{amsfonts,amssymb,latexsym,color,amsmath,enumerate,stmaryrd,amsthm,ushort,dsfont,mathrsfs,authblk}

\usepackage{hyperref} 

\usepackage{graphicx} 

\usepackage{xcolor}
\definecolor{unbleu}{rgb}{0.03, 0.15, 0.4}

\hypersetup{
pdfborder = {0 0 0},
colorlinks,
linkcolor=unbleu,
citecolor=unbleu,
urlcolor=unbleu
}

\usepackage{scalerel} 

\usepackage{geometry}
\usepackage{marginnote}

\usepackage{savesym}
\savesymbol{checkmark}
\usepackage{dingbat}

\numberwithin{equation}{section}

\newcommand{\E}{\mathbb{E}}
\newcommand{\F}{\mathcal{F}}
\newcommand{\X}{\mathcal{X}}
\newcommand{\e}{\operatorname{e}}
\newcommand{\N}{\ensuremath{\mathbb{N}}}    
\newcommand{\R}{\ensuremath{\mathbb{R}}}     
\newcommand{\Z}{\ensuremath{\mathbb{Z}}}    
\renewcommand{\P}{\ensuremath{\mathbb{P}}}

\DeclareMathOperator{\var}{\mathrm{var}}
\DeclareMathOperator{\Var}{Var}

\DeclareMathOperator{\osc}{\mathrm{osc}}
\DeclareMathOperator{\Osc}{Osc}

\DeclareMathOperator{\dd}{\textup{d}\!}

\newtheorem{rem}{Remark}[section]
\newtheorem{lemma}{Lemma}[section]
\newtheorem{defi}{Definition}[section]
\newtheorem{theo}{Theorem}[section]
\newtheorem{ex}{Example}[section]
\newtheorem{prop}{Proposition}[section]
\newtheorem{coro}{Corollary}[section]

\begin{document}


\title{Optimal Gaussian concentration bounds\\ for stochastic chains of unbounded memory}

\author[1]{Jean-Ren\'e Chazottes 
\thanks{Email: jeanrene@cpht.polytechnique.fr}}

\author[2]{Sandro Gallo
\thanks{Email: sandro.gallo@ufscar.br}}

\author[3]{Daniel Y. Takahashi
\thanks{Email: takahashiyd@gmail.com}}

\affil[1]{Centre de Physique Th\'eorique, CNRS, Ecole Polytechnique, Institut Polytechnique de Paris, France}

\affil[2]{Departamento de Estat\'istica, Universidade Federal de S\~ao Carlos (UFSCar), S\~ao Paulo, Brazil}

\affil[3]{Instituto do C\'erebro, Universidade Federal do Rio Grande do Norte (UFRN), Natal, Brazil}

\maketitle

\begin{abstract}  
We obtain optimal Gaussian concentration bounds (GCBs) for stochastic chains of unbounded memory (SCUMs)
on countable alphabets. These stochastic processes are also known as ``chains with complete connections''
or ``$g$-measures''. We consider two different conditions on the kernel: (1) when the sum of its oscillations is less than one, or (2) 
when the sum of its variations is finite, {\em i.e.}, belongs to $\ell^1(\N)$. 
We also obtain explicit constants as functions of the parameters of the model.
The proof is based on maximal coupling. Our conditions are optimal in the sense that we 
exhibit examples of SCUMs that do not have GCB and for which the sum of oscillations is strictly larger than one, or the variation 
belongs to $\ell^{1+\epsilon}(\N)$ for any $\epsilon > 0$.  These examples are based on the existence of phase transitions. We also 
extend the validity of GCB to a class of functions which can depend on infinitely many coordinates.

We illustrate our results by three applications.
First, we derive a Dvoretzky-Kiefer-Wolfowitz type inequality which gives a uniform control on the fluctuations of the empirical 
measure.
Second, in the finite-alphabet case, we obtain an upper bound on the $\bar{d}$-distance between two 
stationary SCUMs and, as a by-product, we obtain new (explicit) bounds on the speed of Markovian approximation in $\bar{d}$. 
Third, we obtain exponential rate of convergence for Birkhoff sums of a certain class of observables.

\medskip


\end{abstract}

\newpage

\tableofcontents

\section{Introduction}

Stochastic chains with unbounded memory (SCUMs) are a natural generalization of Markov chains. Their dynamics is provided by 
a family of probability kernels that describe the probability of observing a symbol at any time given (possibly) the entire past. 
Such processes first appeared in \cite{onicescu/mihoc/1935}, and, since then, have been intensively studied in different fields under different names. In 
the literature of stochastic processes  Doeblin and Fortet \cite{doeblin/fortet/1937} coined the name \emph{chains with complete connections}, while Harris~\cite{harris/1955} later called 
the same objects \emph{chains of infinite order}. In symbolic dynamical systems, stationary SCUMs are studied under the 
name of \emph{$g$-measures} \cite{keane/1972,ledrappier/1974,walters/1975,johansson/oberg/2003}.
In the applied statistics literature, SCUMs have been used to model various natural phenomena, including some popular stochastic 
processes, \emph{e.g.}, categorical time series and binary autoregressive models
\cite{mccullagh/nelder/1983,kedem2005regression,fokianos2018categorical,truquet2019coupling}. SCUMs are also natural 
dynamical counterpart of Gibbs measures on the lattice $\Z$ in statistical physics and the family of probability kernels has 
been called left interval specifications \cite{fernandez/maillard/2005}. Different fields investigated SCUMs using different 
techniques, making this family of stochastic processes a rich object to be studied.

One of the main interests in SCUMs comes from the fact that they exhibit different mixing properties depending on the 
characteristics of the probability kernels. For instance, kernels with strong dependence on the past can have two or more shift 
invariant measures compatible with the kernel
\cite{bramson/kalikow/1993, hulse/2006, gallesco/gallo/takahashi/2014, friedli2015note,  dias_sacha_2016, berger2018non}.
Weak dependence on the past leads to uniqueness of the compatible measure. Different 
uniqueness conditions and the respective mixing properties have been studied
\cite{doeblin/fortet/1937,harris/1955, coelho_quas_1998, bressaud/fernandez/galves/1999a,fernandez/maillard/2005,gallo/garcia/2013,GGT2018}.
  
In the present paper, we investigate the relationship between the characteristics of the probability kernel and the existence
or non-existence of \emph{Gaussian concentration bounds} (GCB)  for the associated SCUMs, \emph{i.e.}, non-asymptotic exponential inequalities for the probability that functions of finite samples deviates from its mean \cite{boucheron2013concentration}. The formal definition of GCBs will be given in Section \ref{sec:GCB}.   We prove that when the kernel has sum of 
oscillations less than one, or has summable variation, the respective SCUMs satisfy a GCB. Moreover, we show that both conditions 
are tight by exhibiting processes that do not satisfy GCB whenever the oscillation is strictly larger than one, 
or the variation belongs to $\ell^{1+\epsilon}(\N)$ for any $\epsilon > 0$. 

We show that a reason for the failure of GCB comes from the non-uniqueness of measures that are compatible with the same 
kernel. Our proof has an interest in its own by providing a method to prove that a GCB cannot be satisfied. 
  
Our bounds are explicit and involve constants that are straightforwardly calculated from the kernels. We apply our inequalities in some 
important examples. Whenever possible, we make comparisons with other papers obtaining GCB for non-independent processes 
in the literature \cite{marton/1998, samson2000, kulske2003, kontorovich2008concentration, gallo2014attractive}. Finally, as a simple application of our 
results, we use the relationship between GCB and transportation cost inequalities to obtain new bounds for $\bar{d}$-distance 
between SCUMs. As a corollary, we obtain the speed of Markovian approximation in $\bar{d}$-distance for SCUMs, in cases not 
covered in \cite{coelho_quas_1998,bressaud_fernandez_galves_1999b, gallo/lerasle/takahashi/2013}.
We also prove a Dvoretzky-Kiefer-Wolfowitz type inequality for SCUMs with summable variation.

Notation and necessary definitions are given in Section \ref{sec:def}. Section \ref{sec:results} contains our main results and in 
Section \ref{sec:consequences} we present some consequences of these results. We end the paper with the proofs of our results in 
Section \ref{sec:proofs}.

\section{{Definitions and notation}}\label{sec:def}

Let $A$ be a countable set (``alphabet'') endowed with the discrete topology. 
We then put the product topology on $A^{\Z}$, the set of bi-infinite sequences drawn from $A$.
We denote by $T$ the shift on $A^\Z$, that is, $(T\omega)_i=\omega_{i+1}$, $i\in\Z$. We equip this space with the
sigma-algebra generated by the cylinder sets $[a_{-n+1},\ldots,a_{n-1}]=\{\omega\in A^\Z : \omega_i=a_i, |i|\leq n-1\}$, $a_i\in A$, $n\in\N$. It comprises all Borel sets of $A^\Z$.

For $i,j \in \Z$ such that $i<j$ we write $\llbracket i, j  \rrbracket = [i,j] \cap \Z$. 
For $i < j$, we indicate the ``string'' $(\omega_i, \ldots, \omega_j)$ by  writing $\omega_i^j$. We also use the convention that
if $i > j$, $\omega_i^j$ is the  empty string. 
When we write $\sigma\in A^{\llbracket i, j  \rrbracket}$, we stress at which coordinate the string $(\sigma_i, \ldots, \sigma_j)$ 
starts. 
When we treat a string of length $k$ as a ``pattern'', regardless of where it is ``located'', we will
simply write $\sigma\in A^k$.\newline
Define $\X^- = A^{\llbracket -\infty,-1 \rrbracket}$.
For $x \in \X^-$, $n \geq 0$,  and $\sigma \in A^{\llbracket 0,n\rrbracket}$, $z= x\sigma$ is a concatenation of the respective symbols,
such that $(\ldots, z_{-1}, z_{0}, \ldots, z_{n}) = (\ldots, x_{-1}, \sigma_0, \ldots, \sigma_n)$.
For all $S\subset\Z$ and $\sigma\in A^S$ we define the projection function associated to all indices $i, j\in S$,  $i\leq j$,  by $
\pi_i^j(\sigma)=\sigma_i^j $.

Throughout the paper $x,y,z$ will denote left-infinite sequences and $\omega$ and $\eta$ will denote right- 
(or bi-) infinite sequences.

\subsection{Kernels and SCUMs}

To define the probability measures of interest in this paper, namely stochastic chains of unbounded memory (SCUMs, for short), we first need to define 
what we mean by a probability kernel. 

\begin{defi}[Probability kernel]
For all $n \in \Z$, a probability kernel $g_n$ is a function $g_n: A \times A^{\llbracket -\infty, n-1\rrbracket} \to [0,1]$ such that for all 
$x \in A^{\llbracket -\infty,n-1 \rrbracket}$, $\sum_{s \in A}g_n(s|x) = 1$. Because we will only consider shift-invariant kernels, with 
some abuse of notation, we will always refer to function $g$ instead of $g_n$ regardless of the index set inside the function. 
\end{defi}

Let us first introduce SCUMs started from a fixed past.

\begin{defi}[Probability measure started with a fixed past]\label{def:fixedpastproba}
For $x \in \X^-$, $k\ge-1$ and $\sigma \in A^{\llbracket 0,\ldots, k\rrbracket}$, we define $P^{x\sigma}$ as the probability measure  {specified 
by $g$ when} started with $x\sigma \in A^{\llbracket -\infty, k\rrbracket}$, that is,   {for all $\omega\in A^\N$ and all $n\ge k+1$}
\[ 
P^{x\sigma}({[\omega_{k+1}^n]}) = \prod_{j=k+1}^n g\big(\omega_j|x\sigma \omega_{k+1}^{j-1}\big).
\]
Sometimes we will write $P^{x\sigma}_g$ when it is not clear from the context to which kernel the measure corresponds. 
\end{defi}

Now, we introduce the definition of SCUMs compatible with a kernel, which is similar to the definition of a Gibbs measure compatible with a specification. 
We denote by $\F_i^j$ the $\sigma$-algebra generated by the cylinders with base in the interval $\llbracket i, j  \rrbracket$. We use the shorthand 
notation
$\F_k = \F_0^k$, $k \geq 0$. 
\begin{defi}[Probability measure compatible with a kernel]\label{def:compa}
We say that a probability measure $\mu$ on $A^\Z$ is compatible with $g$ if, for all $n\in\Z$, $a \in A$ and $\mu$-a.e. $x\in 
A^{\llbracket -\infty,n-1 \rrbracket}$, we have
\[
\mu([a]|\F_{-\infty}^{n-1})(x) = g(a|x).
\] 
\end{defi}
A stationary stochastic process $(X_n)_{n\in\Z}$, where the random variables take values in $A$, is charaterized by a
shift-invariant probability measure $\mu$ on $A^\Z$, that is, a measure satisfying $\mu\circ T^{-1}=\mu$.
The canonical process $(X_n)_{n\in\Z}$ corresponding to a measure $\mu$ compatible with a kernel is called a  stochastic chains of unbounded memory 
(SCUM) compatible 
with $g$. Equivalently, we say that a SCUM $(X_n)_{n\in\Z}$ is compatible with $g$ if it satisfies
\[
\E_\mu[\mathds{1}_{a}(X_n)|X^{n-1}_\infty=x]=g(a|x)
\]
for all $n$, $a\in A$, and $\mu$-a.e $x\in A^{\llbracket -\infty,n-1 \rrbracket}$, where $\mathds{1}_{a}(\cdot)$ is the indicator function of the symbol $a$.

\subsection{Regularity assumptions on kernels}

In order to study the statistical properties of SCUMs we will quantify how the kernel $g$ depends on the past in two ways. We will use the \emph{oscillation} of $g$ of order $j 
\geq 1$, defined by
\[
\Osc_j (g) := \sup \left\{ \frac{1}{2}\sum_{a \in A} |g(a|z)-g(a|z')|: z,z' \in \X^-, z_k = z'_k, \forall k \neq -j \right\}.
\]
and the variation of order $j \geq 1$, defined by 
\[
\Var_j (g) := \sup \left\{ \frac{1}{2}\sum_{a \in A} |g(a|z)-g(a|z')|: z,z' \in \X^-, z_k = z'_k, \forall k \geq -j \right\}.
\]
We also  define $\Var_0(g):=\sup_{z,z' \in \X^-} \frac{1}{2}\sum_{a \in A} |g(a|z)-g(a|z')|$.
Note that the usual definition of oscillation (\cite[for instance]{hulse/2006, fernandez/maillard/2005}) and variation
(\cite[for instance]{harris/1955,keane/1972}) are, respectively,
\[
\osc_j (g) := \sup \left\{ |g(a|z)-g(a|z')|: a\in A, z,z' \in \X^-, z_k = z'_k, \forall k \neq -j \right\}
\]
and
\[
\var_j (g) := \sup \left\{  |g(a|z)-g(a|z')|: a \in A, z,z' \in \X^-, z_k = z'_k, \forall k \geq -j \right\}.
\]
When the alphabet is finite the definitions are equivalent since we have $\var_j(g) \leq \Var_j (g) \leq |A|\var_j(g)$
and $\osc_j(g) \leq \Osc_j (g) \leq |A|\osc_j(g)$.
We use $\Osc_j(g)$ and $\Var_j(g)$ as these quantities appear naturally in the proofs when we introduce maximal coupling and they are more convenient to state our results when $|A| = \infty$.
Given a kernel $g$, the following quantities will play a central role: 
\begin{equation}\label{def:delta}
\Delta (g):= 1- \sum_{j = 1}^\infty \Osc_j (g)
\end{equation} 
and 
\begin{equation}\label{def:gamma}
\Gamma (g):= \prod_{j=0}^\infty (1-\Var_j (g)).
\end{equation} 
\begin{rem}[Relation with existence/uniqueness criteria of the literature]
A natural question to ask is whether there exists a unique shift-invariant measure compatible with a given kernel $g$. If $\Delta(g) > 0$, Theorem 4.6 in 
\cite{fernandez/maillard/2005} states that there is at most one compatible measure, which is therefore  shift-invariant. In the case 
of finite alphabet, the assumption $\Gamma(g)>0$  implies uniqueness of a shift-invariant compatible measure  (see \cite{harris/1955,keane/1972}
for instance). In the case of countably infinite alphabets, the conditions for uniqueness of the compatible measure are not based on $\var_j(g)$ anymore, 
and it is not obvious how to compare the assumption $\Gamma(g)>0$ with other assumptions of the literature. For our purpose, we only discuss the 
conditions of uniqueness and existence when needed in the proofs. 
\end{rem}

\subsection{Gaussian concentration bound}\label{sec:GCB}

We first define a class of functions.
Let $n \geq 0$ and $f: A^{n+1} \to \R$.  Since $A$ is countable and endowed with the discrete topology, $f$ is continuous.
Define 
\[
\delta_j (f) =
\sup\big\{\big|f(\omega_0^{j-1}a\,\omega_{j+1}^n)-f(\omega_0^{j-1}b\,\omega_{j+1}^n)\big|: a,b \in A, \omega \in A^{n+1}\big\}
\]
for $0\leq j \leq n$.  
\begin{defi}
\label{def-all-loc-functions}
Let
\[
\mathcal{L}=\bigcup_{n\geq 0}\mathcal{L}_n\quad\text{where}\quad \mathcal{L}_n=\big\{f: A^{n+1} \to \R: \delta_j(f)<+\infty, j=0,\ldots,n\big\}.
\]
\end{defi}
Each $f\in\mathcal{L}$ is bounded. Indeed, for each $f\in\mathcal{L}$, there exists $n$ such that $f\in\mathcal{L}_n$.
Now pick an arbitrary $\omega'\in A^{n+1}$.
An obvious telescoping then gives $|f(\omega)-f(\omega')|\leq \sum_{j=0}^n \delta_j(f)$, whence
$\|f\|_\infty\leq |f(\omega')|+\sum_{j=0}^n \delta_j(f)<+\infty$.
We denote by $\ushort{\delta}(f)$ the column vector  of size $n+1$ whose $j$-th coordinate is $\delta_{j-1}(f)$. 
For a function $f:A^\N\to\R$, we define the semi-norm
\begin{equation}\label{def-delta2}
\|\ushort{\delta}(f)\|^2_2 =  \sum_{j =0}^\infty \delta_j(f)^2.
\end{equation}
If $f\in\mathcal{L}_n$, we have $\|\ushort{\delta}(f)\|^2_2 =  \sum_{j =0}^n \delta_j(f)^2$, in which case  $\|\ushort{\delta}(f)\|^2<+\infty$.

For every interger $p\geq 1$ and $\ushort{v}=(v_{0},v_{1},\ldots)$ with $v_i\in\R$, define
\[
\|\ushort{v}\|_p^p=\sum_{j=0}^\infty |v_j|^p\,.
\]
We say that $\ushort{v}\in\ell^p(\N)$ if $\|\ushort{v}\|_p<+\infty$.\newline
For $f:A^{\Z}\to\R$ $\mu$-integrable, we use the notation $\E_\mu[f]=\int f \dd\mu$.

We can now define what we mean by a Gaussian concentration bound.
\begin{defi}[Gaussian concentration bound]\label{def-GCB}
A probability measure $\mu$ on $A^{\Z}$ or on $A^{\N}$ is said to satisfy a Gaussian concentration bound (GCB for short) if there 
exists a constant $C>0$ such that, for all $f\in\mathcal{L}$, we have
\begin{equation}\label{eq-GCB}
\E_\mu\left[\e^{f  -\E_{\mu}[f]}\right] \leq \e^{C \|\ushort{\delta}(f)\|^2_2}
\end{equation}
where $\|\ushort{\delta}(f)\|_2$ is defined in \eqref{def-delta2}.
\end{defi}

A key-point in this definition is that $C$ does neither depend on $n$ nor on $f$.
Since $f$ is bounded, this inequality implies that, for all $\theta\in\R$, we have
\[
\E_\mu\left[\e^{\theta(f  -\E_{\mu}[f])}\right] \leq \e^{C\theta^2 \|\ushort{\delta}(f)\|^2_2}
\]
and using a standard argument (usually referred to as Chernoff bounding method, see \cite{mcdiarmid_1989}), we deduce that, for
all $u>0$, 
\begin{equation}\label{Chernoff-GCB}
\mu(\left |f  -\E_{\mu}[f] \right| > u)\leq 2\exp \left( -\frac{u^2}{4C\|\ushort{\delta}(f)\|^2_2}\right).
\end{equation}

The formulation of Definition \ref{def-GCB} is made in such a way that we can take a probability measure with a fixed past (see Definition \ref{def:fixedpastproba}). Also, if we have a shift-invariant probability measure $\mu$, then it is indifferent to work either with $A^{\Z}$ or 
$A^{\N}$.

\section{Main results and examples}\label{sec:results}

\subsection{GCB under a condition on the oscillation of the kernel}

Our first result is a GCB for a probability measure started with a fixed past in the sense of Definition \ref{def:fixedpastproba}. Note that the bounds are uniform in the past $x \in \X^-$. 

\begin{theo} \label{theo:dobrushinconcentration}
Let $g$ be a kernel such that $\Delta(g) > 0$.
Then, for all $f\in\mathcal{L}$ and $\theta \in \R$, we have
\begin{equation}\label{gcb-Delta}
\sup_{x\in \X^-} \E_{P^x}\left[\e^{\theta (f  -\E_{P^x}[f])}\right] \leq \e^{\frac{\theta^2 \Delta(g)^{-2}}{8} \|\ushort{\delta}(f)\|^2_2}.
\end{equation}
As a consequence, for all $u > 0$, we have
\begin{equation}\label{dev-ineq-Delta}
\sup_{x\in \X^-} P^x(\left |f  -\E_{P^x}[f] \right| > u)\leq 2\exp \left( -\frac{2 u^2}{\Delta(g)^{-2} \|\ushort{\delta}(f)\|^2_2}\right).
\end{equation}
\end{theo}
Let us illustrate this theorem with two examples.
\begin{ex}[Binary autogregressive process] \label{ex:biautogressive}
Consider a function $\psi: \R \to (0,1)$ such that $\psi(r) + \psi(-r) = 1$ and an absolutely summable sequence of real numbers
$(\xi_j)_{j \geq 0}$. Then the kernel $g:\{-1,+1\} \times \{-1,+1\}^{\llbracket -\infty, -1\rrbracket} \to (0,1)$ is defined as
\[
g(a|x) = \psi\left(a\sum_{j=1}^\infty \xi_j x_{-j} + a\xi_0\right).
\]
The process generated by this kernel is called a binary auto-regressive process \cite{kedem2005regression}. If $\psi$ is 
differentiable, we have that $\Osc_j (g) \leq  2(\sup\psi')|\xi_j|$, hence we have $\Delta (g) \geq 1-2(\sup\psi')
\sum_{j=1}^\infty |\xi_j|$. For instance, if $\psi(u)=(1+\e^{-2u})^{-1}$ then $\Delta (g)\ge 1-\sum_{j=1}^\infty |\xi_j|$.
\end{ex}

\begin{ex}[Poisson regression for count time series] 
Let $A=\N$ and $(\xi_j)_{j \geq 0}$ be a sequence of non-positive absolutely summable real numbers, and a constant $c > 0$.
For all $x \in \N^{\llbracket -
\infty, -1\rrbracket}$, let
\[
v(x) = \exp\left(\sum_{j=1}^\infty \xi_j \min\{x_{-j}, c\}\right). 
\]
For all $a\in \N$ and $x \in \X^-$, the kernel of a Poisson 
regression model  is defined as \cite{kedem2005regression}
\[
g(a|x) = \frac{\e^{-v(x)}v(x)^a}{a !}.
\]
Applying the mean value theorem to $\psi(r) = \e^{-\e^{r}}\e^{ra}/a!$, and maximizing on $r \in (-\infty,0]$ for each $a \in \N$, we 
obtain $\Osc_j (g) \leq  \e^{-1}\sum_{a\ge0}\frac{1}{a!}|\xi_j|=|\xi_j|$. Therefore, $\Delta (g) \geq 1- \sum_{j=1}^\infty|\xi_j|$. 
\end{ex}

We also have a theorem for stationary SCUMs.
\begin{theo} \label{coro:dobrushinconcentration}
If $\mu$ is a shift-invariant measure compatible with a kernel $g$ satisfying
$\Delta (g)  > 0$, then inequalities \eqref{gcb-Delta} and \eqref{dev-ineq-Delta} hold with $\mu$ in place of $P^x$, with the same 
constant.
\end{theo}

\subsection{GCB under a condition on the variation of the kernel}

We have the analog of Theorem \ref{theo:dobrushinconcentration} under a natural condition on the variation.  The bounds are uniform in the past $x \in \X^-$.
\begin{theo} \label{theo:summableconcentration}
Let $g$ be a kernel such that $\Gamma (g) > 0$.
Then, for all $f\in\mathcal{L}$ and $\theta \in \R$, we have
\begin{equation}\label{gcb-Gamma}
\sup_{x\in \X^-} \E_{P^x}\left[\e^{\theta (f  -\E_{P^x}[f])}\right] \leq \e^{\frac{\theta^2 \Gamma (g)^{-2}}{8} \|\ushort{\delta}(f)\|^2_2}.
\end{equation}
As a consequence,  for all $u > 0$, we have
\begin{equation}\label{dev-ineq-Gamma}
\sup_{x\in \X^-} P^x(\left |f  -\E_{P^x}[f] \right| > u)\leq 2\exp \left( -\frac{2 u^2}{\Gamma (g)^{-2} \|\ushort\delta(f) \|^2_2}\right).
\end{equation}
\end{theo}
We give a class of examples illustrating this theorem.
\begin{ex}[{Convex} mixture of Markov chains]
Let $(\lambda_j)_{j \geq 1}$ be a sequence of non-negative real numbers such that $\sum_{j=1}^\infty \lambda_j = 1$. Let $A$ be a 
countable set. Define a family of Markov kernels $p^{[k]}:A \times A^{\llbracket -k, -1\rrbracket} \to [0,1], k\ge0$,  that is, for all
$x \in \X^-$, $\sum_{a \in A} p^{[k]}\big(a|x^{-1}_{-k}\big) = 1$. The kernel for mixture of Markov chains is defined, for all $a 
\in A$ and $x \in \X^-$, as
\[
g(a|x) = \sum_{j = 1}^\infty \lambda_j\, p^{[j]}\big(a|x^{-1}_{-j}\big).
\]
We have $\sum_{j=1}^\infty \Var_j (g) \leq \sum_{j =1}^\infty j \lambda_j. $ This result is quite general 
since a large class of kernels, including all kernels  $g$ on finite alphabet with $\lim_j\Var_j(g) = 0$, can be represented
as a  {convex} mixture of Markov chains \cite{kalikow}.
\end{ex}

The next result complements Theorem \ref{theo:summableconcentration} in the case of stationary SCUMs.
\begin{theo} \label{coro:summableconcentration}
If $\mu$ is a shift-invariant measure compatible with a kernel $g$ such that 
$\Gamma (g)>0$,  then inequalities \eqref{gcb-Gamma} and \eqref{dev-ineq-Gamma} hold with $\mu$ in place of $P^x$, with the same constant.
\end{theo}

\subsection{{Optimality of the bounds}}\label{sec:optbound}

Here we show that Theorems \ref{coro:dobrushinconcentration} and \ref{coro:summableconcentration} are optimal already for 
binary alphabets.  
Theorems \ref{theo:PTnoGCB} and \ref{theo:renewal} below give necessary conditions to get GCB for a large class of processes 
that exhibit phase transition. Our optimality results are simple consequences of these theorems. 

The following result shows that, for kernels satisfying strong regularity conditions, a ``phase transition'' is a fundamental obstruction
for having GCB.
\begin{theo} \label{theo:PTnoGCB}
Let $g$ be a kernel such that $\inf_{a\in A, x\in \X^-} g(a|x)> 0$ and $\lim_j\Var_j(g) = 0$. 
If $g$ has  two (or more) distinct ergodic compatible measures, then they do not satisfy GCB. 
\end{theo} 

It is proved in \cite{hulse/2006} that, for all $\epsilon > 0$, there are examples of $g$ such that $\Delta (g) + \epsilon < 0$ and exhibit 
multiple shift-invariant ergodic compatible measures. Because of Thereom \ref{theo:PTnoGCB}, this implies that the shift-invariant ergodic compatible 
measures do not satisfy GCB. This shows optimality of Theorem \ref{coro:dobrushinconcentration} regarding the assumption $\Delta (g)  > 0$, a fact 
that we now state as a corollary of Theorem \ref{theo:PTnoGCB}. 

\begin{coro} 
For any $\epsilon >0$, there is a kernel $g$ on a binary alphabet and a compatible shift-invariant probability measure $\mu$ such 
that $\sum_{j = 1}^\infty  \Osc_j (g) \in \left(1, 1+\epsilon \right]$ and $\mu$ does not satisfy GCB. Moreover, $g$ can be chosen to 
satisfy $\lim_j \Var_j(g) = 0$ and $\inf_{a \in A,x\in \X^-} g(a|x) > 0$.
\end{coro}

To demonstrate the optimality of Theorem \ref{coro:summableconcentration}, we consider ``renewal measures'', 
a particular class of  SCUMs.
Let $(q_j)_{j \geq 0}$ with $q_j \in (0,1)$. Given $x \in \X^-$, let $\ell(x) = \inf\{k \geq 0: x_{-k-1} = 1\}$ and $\ell(\ldots00) = \infty$. 
We define the renewal 
kernel $\tilde g:\{0,1\}\times\{0,1\}^{\llbracket -\infty, -1\rrbracket} \to (0,1)$ by taking $
\tilde g(1|x) = q_{\ell(x)}$. Obviously, if $q_\infty=0$ then the degenerate measure $\delta_{0^\infty}$ is stationary and compatible, 
and trivially satisfies GCB. However, we call renewal measure the stationary measure $\tilde\mu$ compatible with $g$
satisfying $\tilde\mu([a])>0$ for any $a\in \{0,1\}$, when it exists. It is not difficult to see that this measure will 
actually consists of a sequence of i.i.d. concatenation of blocks of the form $0^i1,i\ge1$. The probability that the distance between 
two consecutive $1$'s equals $n\ge1$,  denoted $f_n$, is
\begin{equation}\label{eq:fn}
f_n:=P_{\tilde g}^{x1}(0^{n-1}1)=q_{n-1}\prod_{i=0}^{n-2}(1-q_i),\quad \forall n\ge1,\forall x\in\X^-
\end{equation}
with the convention $\prod_{i=0}^{-1}=1$. The probability distribution $(f_n)_{n\ge1}$ is usually called inter-arrival distribution in the 
literature.  Then, the renewal measure exists if and only if the expected distance between consecutive ones, $\sum_{n\ge1}nf_n$, 
is finite, which is equivalent to 
\begin{equation}\label{eq:existsrenewal}
\sum_{j\ge1}\prod_{i=0}^{j-1}(1-q_i)<\infty.
\end{equation}  We have the following result. 
\begin{theo}\label{theo:renewal}
The renewal measure $\tilde\mu$ satisfies a GCB if, and only if, $\sum_n f_n r^n<\infty$ for some $r>1$.
\end{theo}

Consider now the particular case in which  $q_j=j^{-\alpha}$ for $j\ge2$ with $\alpha\in (0,1)$ so that \eqref{eq:existsrenewal} is satisfied and therefore 
the renewal process exists. A simple calculation shows that in this case $f_n$ is stretched exponential, and therefore, by Theorem \ref{theo:renewal}, the 
renewal process does not satisfy GCB. In order to fix ideas, let us put $q_0=2/3$ and $q_1=1/2$ and $q_\infty=0$. It is easy to check that $\Var_j (g) = q_j<1,j\ge0$ and in this case, $\Gamma(g)>0$ is equivalent to $\sum_{i\ge0}\Var_i(g)<\infty$. Hence, if we choose $\alpha = (1+\epsilon/2)(1+\epsilon)^{-1}$ the variation will not be summable,
but $\sum_{j=1}^\infty\Var_j (g)^{1+\epsilon} < \infty$, proving that Theorem \ref{coro:summableconcentration} is optimal, a fact that we state as a 
corollary of Theorem \ref{theo:renewal}.

\begin{coro} 
For any $\epsilon >0$, there is a kernel $g$ on a binary alphabet and a compatible shift-invariant probability measure $\mu$ such 
that 
\[
\sum_{j = 1}^\infty \Var_j (g)^{1+\epsilon} < \infty, \; \sum_{j = 1}^\infty \Var_j (g) = \infty
\]
and $\mu$ does not satisfy GCB. Moreover, $g$ can be chosen to satisfy
$\inf_{x \in \X^-} g(a|x) > 0$ for some $a \in A$.
\end{coro}

\begin{rem}
Note that the two kernels used to obtain examples of processes which do not satisfy GCB exhibit phase transition since for the renewal process, when 
$q_j=j^{-\alpha}$ we have that $q_\infty=0$ (to get $\Var_k(g)\rightarrow0$), in which case the Dirac measure $\delta_{0^\infty}$ is also compatible. 
However, this kernel does not fall into the class considered by Theorem \ref{theo:PTnoGCB} because it does not satisfy $\inf_{a\in A,x \in \X^-} g(a|x) > 
0$. 
\end{rem}
\begin{rem}
Kernels satisfying $\inf_{a\in A,x \in \X^-} g(a|x) > 0$ are  said to be ``strongly non-null''.  
If we restrict to strongly non-null kernels $g$ instead of assuming the weak non-nullness $\sum_a\inf_{x\in\X^-} g(a|x) > 0$, then we don't know if the summable variation condition is tight for the validity of GCB. Nevertheless, 
even if we restrict to strongly non-null kernels $g$,  {GCB does not hold} in general beyond square 
summable variation because of Theorem \ref{theo:PTnoGCB}  and the existence of examples with phase transition for 
strongly non-null kernels such that $\sum_{j = 1}^\infty \Var_j (g)^{2+\epsilon} < \infty, \epsilon > 0$, see \cite{berger2018non}.
\end{rem}

\subsection{GCB for a more general class of functions}

Denote by $\mathrm{C}(A^{\N})$ the set of real-valued continuous functions on $A^{\N}$ that we equip with the supremum norm. 
We define two of its subspaces, namely the set of bounded continuous functions, denoted $\mathrm{BC}(A^{\N})$, and the set of 
uniformly continuous functions, denoted $\mathrm{UC}(A^{\N})$. As $A^{\N}$ is in general not compact, $\mathrm{UC}(A^{\N})$ 
intersects but does not contain $\mathrm{BC}(A^{\N})$, nor does $\mathrm{BC}(A^{\N})\supseteq \mathrm{UC}(A^{\N})$.
Obviously, the set $\mathcal{L}$ (see Definition \ref{def-all-loc-functions}) is contained in each of these three spaces.
We have $\mathrm{C}(A^{\N})=\mathrm{UC}(A^{\N})=\mathrm{BC}(A^{\N})$ if and only if $A^{\N}$ is compact, which holds if and 
only if $A$ is finite.
For $f\in \mathrm{C}(A^{\N})$, 
let
\[
\mathrm{var}_n(f)=\sup\{|f(\omega)-f(\omega')|: \omega_i=\omega'_i, i=0,\ldots,n\}\,, n\geq 0.
\]
One can easily check that $\mathrm{var}_n(f)\to 0$ if and only if $f\in \mathrm{UC}(A^{\N})$.

We can generalize Theorems \ref{theo:dobrushinconcentration}, \ref{coro:dobrushinconcentration}, 
\ref{theo:summableconcentration} and \ref{coro:summableconcentration}, thanks to the following abstract result.

\begin{theo} \label{concentrationformoregeneralfunctions}
If a probability measure $\mu$ satisfies a Gaussian concentration bound for some constant $C>0$, then 
this bound remains true for all $f \in \mathrm{UC}(A^{\N})\cap \mathrm{BC}(A^{\N})$ such that 
$\|\ushort{\delta}(f)\|_2<+\infty$, with the same constant $C$.
\end{theo}

We refer the reader to  Section \ref{sec:app_infinite}  for a natural  application of Theorem \ref{concentrationformoregeneralfunctions}.

\subsection{{Comparisons with existing results}}\label{remarksonmainresults}

We first mention that, in the independent case, $\Gamma(g) = \Delta(g) = 1$, so we recover McDiarmid's inequality with the optimal 
constant \cite{mcdiarmid_1989}.
Next, for the sake of clarity, we will distinguish between the one-step Markov case and the non-Markov case.

\subsubsection{The Markov case}

In the one-step Markov case $\Gamma (g) =1-\Var_0(g) = 1-\Osc_1(g) = \Delta (g)$. 
In order to compare our results with those in the literature, let us make a slight abuse of notation and put $g=Q$ where
$Q:A\times A\rightarrow[0,1]$ is the transition matrix  defined by $g$ through $Q(a|b):=g(a|x)$ for any $x\in \mathcal X^-$ such that 
$x_{-1}=b$. Now, observe that 
\begin{equation}\label{eq:doeblin}
\Var_0(g)=\sup_{a,b\in A}\|Q(\cdot|a)-Q(\cdot|b)\|_{{\scriptscriptstyle \mathrm{TV}}}=:d(Q).
\end{equation}
This is the  Dobrushin ergodicity coefficient (see \cite{dobrushin1956central} and \cite[Section 18.2]{douc2018markov} for 
instance). Therefore, our theorems state that, if $d(Q)<1$ we have
\[
\mu(\left |f  -\E_{\mu}[f] \right| > u)\leq 2\exp \left( -\frac{2 u^2}{(1-d(Q))^{-2}\|\ushort\delta(f) \|^2_2}\right).
\]
GCBs for Markov chains on countable alphabet under $d(Q)<1$ were previously obtained by other authors \cite{kontorovich2008concentration, marton1996bounding}. The literature on GCB for Markov chains is extensive and the interested reader should consult \cite{paulin2015concentration} for a nice review. Here, we make a brief comment on the limitation of our result in the Markov case.
Without getting into details, let us mention that \emph{uniform ergodicity} of a Markov chain is equivalent to  $d(Q^m)<1$
for some $m\ge1$ \cite[Section 18.2]{douc2018markov}. So our condition is slightly stronger than assuming uniform ergodicity. 
Now, uniform ergodicity is stronger than \emph{geometric ergodicity}, which is necessary and sufficient for having a Gaussian 
concentration bound for a Markov chain \cite{dedecker/gouezel/2015}. 
As far as we know, the only explicit bound under geometric ergodicity was recently obtained in \cite{havet2020quantitative}. 

For a uniformly ergodic Markov chain, Corollary 23.2.4 in \cite{douc2018markov} states that, if
\[
{\mathbf d}:=\sum_{m\ge1}d(Q^m)<\infty
\]
then
\begin{equation}\label{eq:doucetal}
\mu(\left |f  -\E_{\mu}[f] \right| > u)\leq 2\exp \left( -\frac{2 u^2}{(1+{\mathbf d})^{2}\|\ushort\delta(f) \|^2_2}\right),\;u>0.
\end{equation}
We note that we obtain the same inequality using our Theorem \ref{theo:basicGCB} below for the case of Markov chains. When $d(Q)<1$ we have that $d(Q^m)\le d(Q)^m$ \cite[Equation 18.2.2]{douc2018markov} and obtain the inequality
$1+{\mathbf d}\le (1-d(Q))^{-1}$, which permits to compare the two previous bounds. The main advantage of the bound 
\eqref{eq:doucetal} is that it might hold even if $d(Q)=1$ but $d(Q^m)<1$ for some $m\ge1$.
Furthermore, there is in principle a slight advantage, even in the case where $d(Q)<1$, on having
$1+{\mathbf d}$ instead of $(1-d(Q))^{-1}$. However, calculating $d(Q^m)$ is intractable even for large but finite alphabets. 
Finally, let us also mention Theorem 7.1 of \cite{redig2009concentration} is also based on a coupling 
approach, but the obtained constant is not optimal, even for nice particular cases. We point out that 
in \cite{redig2009concentration} Theorem 7.1 the constant $C$ should appear in the denominator of the 
quotient appearing in the bound and in the example of the house-of-cards process, in the uniform case $C$ should be equal to 
$1/2q$ instead of $1/2(1-q)$.

\subsubsection{The non-Markov  case}

In general, $\Gamma (g) \neq \Delta (g)$ and our conditions are complementary to each other, as we now illustrate.
Consider the kernel in Example \ref{ex:biautogressive} with $\psi(u)=1/(1+\exp(-2u))$, $\xi_i \geq \xi_j \geq 0$ for all $j > i \geq 1$,
and $\xi_0 = \sum_{k = 1}^\infty \xi_k$. 
We have
\[
\frac{2\e^{2\xi_j}}{(1+\e^{2\xi_j})^2}\, \xi_j \leq \Osc_j (g) \leq  \xi_j
\quad\text{and}\quad
\frac{2\e^{2\xi_1}}{(1+\e^{2\xi_1})^2}  \sum_{k>j}\xi_k \leq \Var_j (g) \leq \sum_{k>j}\xi_k.
\]
If $\xi_j = c/j^{1+\epsilon}$ with $\epsilon \in (0,1]$ and small enough $c > 0$, we have $\Delta (g) > 0$ but $\Gamma (g) = 0$.
On the other hand, if $\xi_j = C/j^{1+\epsilon}$ with $\epsilon \in (1,\infty)$ and large enough $C$ we have
$\Gamma(g)>0$, but $\Delta (g) < 0$.

There are SCUMs that satisfy GCB, but are not covered by our results.  Here is an example using the renewal kernel defined in 
Section \ref{sec:optbound}. 
For this example, let $\alpha \in (0,1)$ and consider the sequence $q_j = q_\infty +  \alpha/j^{\alpha}$. For the renewal kernel we 
have $\Osc_j (\tilde g) = \Var_j (\tilde g) = q_j - q_\infty = \alpha/j^{\alpha}$. {If $q_\infty > 0$ we easily find that the inter-arrival 
distribution defined in \eqref{eq:fn} is exponential, and therefore, by Theorem \ref{theo:renewal}  this renewal process satisfies GCB. 
However, neither the oscillation nor the variation are summable, and therefore we have  $\Delta (\tilde g) < 0$ and
$\Gamma (\tilde g) = 0$.} The problem in this example is that we have slow uniform variation rate. The interested reader should consult
\cite{gallo/2011,gallo/paccaut/2013,gallo/garcia/2013}.

Marton \cite{marton/1998} proved a property which is equivalent to a version of Theorem 
\ref{coro:summableconcentration} in which $\|\ushort{\delta}(f)\|_2^2$ is substituted by $n\|\ushort{\delta}(f)\|_{\infty}^2$.
Because $\|\ushort{\delta}(f)\|_2^2 \leq (n+1)\|\ushort{\delta}(f)\|_{\infty}^2$, our result gives a slight improvement.
For example, consider $A=\{0,1\}$, $\epsilon > 0$, and $f(x_0^n) = \sum_{j=0}^n x_j/(j+1)^{(1+\epsilon)/2}$. In this case, for all 
$n\geq 0$, $\|\ushort{\delta}(f)\|_2^2 <C$ for some constant $C$, but $n\|\ushort{\delta}(f)\|_{\infty}^2 = n+1$.
Perhaps more importantly, we offer a different proof. Marton's proof is based on a transportation cost inequality together with its 
tensorization, whereas our proof is based on the martingale method together with coupling inequalities. Because of the stationarity 
requirement in \cite{marton/1998}, we do not know if we can obtain Theorem \ref{theo:summableconcentration} using the same 
method as in \cite{marton/1998}.
We also note that \cite{gallo2014attractive} obtained a version of Theorem \ref{coro:summableconcentration} for finite alphabets by 
a different approach than the one we use here (coupling-from-the-past algorithm), and with a suboptimal constant $2/9$ instead of $2$ as we obtained 
here. 
To conclude, let us mention that \cite{kulske2003} proved a GCB for Gibbs random fields satisfying the two-sided Dobrushin 
condition. On $\Z$, if a Gibbs specification satisfies the two-sided Dobrushin condition then it also satisfies $\Delta(g) > 0$ 
\cite[Theorem 4.20]{fernandez2004chains}. 
Therefore, our Theorem \ref{coro:dobrushinconcentration} 
implies the result in \cite{kulske2003} for one-dimensional Gibbs measures, but we do not know whether the converse also holds.

\subsection{Some open problems}

There are only a few results in the literature giving necessary conditions for the existence of GCB for dependent process. We think that answers to  the following questions could help in the development of new tools to prove necessary conditions for GCB.

\begin{itemize}
\item
Is there a SCUM on finite alphabet with a unique compatible stationary measure but which does not satisfy GCB?

\item
Do we have GCB when $\Delta (g) = 0$ and $\Var_j (g) =O(1/j)$? 

\item What is the ``rate of concentration'' for kernels with more than one compatible measure? 

\end{itemize}

\section{Applications}\label{sec:consequences}

In this section we explore some consequences of our results. We show (1) 
a new bound on the probability of deviation of the empirical distribution from the stationary distribution,
(2) new bounds for the distance between two processes under the $\bar d$-distance, and (3) an exponential bound for the rate convergence of
Birkhoff sums for SCUMs.  For further applications of GCBs in general, the reader can check \cite{CCR2017}.

\subsection{Dvoretzky-Kiefer-Wolfowitz type inequality}

In statistics we are often interested in the empirical distribution. For $\sigma \in A^{\llbracket 1, k\rrbracket}$
and $\omega\in A^{\N}$, let
\[
\hat{\rho}_{n,k}(\sigma,\omega) = \frac{1}{n-k+2}\sum_{j = 0}^{n-k+1} \mathds{1}_{\sigma}(\pi_j^{j+k-1} \omega).
\]
We will simply write $\hat{\rho}_{n,k}(\sigma)$ for the corresponding random variable.
To estimate the probability of deviation from the expected value, it is natural to use Theorem \ref{theo:summableconcentration}
to obtain
\[ 
\mu\left(|\hat{\rho}_{n,k}(\sigma)-\rho(\sigma)|>  u\right)  \leq 2\exp\left(-(n-k+2) \Gamma (g)^{2} u^2\right)
\]
where $\hat{\rho}_{n,k}(\sigma)=\hat{\rho}_{n,k}(\sigma,\cdot)$ and
$\rho(\sigma) := \E_\mu[\hat{\rho}_{n,k}(\sigma)] = \mu([\sigma])$.
If we want to obtain a uniform bound, we should upper bound 
\[
\mu\left(\|\hat{\rho}_{n,k}-\rho\|_\infty>u\right)=\mu\left(\sup_\sigma|\hat{\rho}_{n,k}(\sigma)-\rho(\sigma)|>u\right),\; u>0.
\]
In this case, it is tempting to use a union bound. However, when the cardinality of the set of symbols is large, we get a bad bound, 
and when $A = \N$ this approach obviously fails. One possible solution is to concentrate directly the uniform deviation
$\|\hat{\rho}_{n,k}-\rho\|_\infty $, which yields the following result.

\begin{theo} \label{theo:DKW}
Let $g$ be a kernel and $\mu$ be a shift-invariant measure compatible with $g$. 
If $\Gamma (g)>0$, we have, for all $u > 0$ and for all $n > 0$ and $0< k\leq n$,
\begin{equation}\label{ineq:DKW}
\mu\left(\|\hat{\rho}_{n,k}-\rho\|_\infty > \frac{u+\sqrt{2k}}{\sqrt{(n-k+2)\Gamma (g)}}\right)  \leq \exp\big(-\Gamma (g)\,u^2\big). 
\end{equation}
\end{theo}
A similar result for $k = 1$ was obtained in \cite{kontorovich2014uniform}  for Markov chains and hidden Markov models, but as far 
as we know, our result is the first in the literature for SCUMs. {Because Theorem \ref{theo:DKW} gives a uniform control on the 
empirical distributions, we can use these results to estimate quantities that can be written as functionals of empirical distribution, 
\emph{e.g.}, entropy, kernels, and potentials of the 
processes.}

\subsection{Explicit upper bound for the $\bar{d}$-distance, and speed of Markovian approximation}

Given two probability measures $\mu$ and $\nu$ on $A^\Z$, a coupling of $\mu$ and $\nu$ is a probability measure $\P$ on
$A^\Z\times A^\Z$ satisfying, for all $B\in\mathcal{F}_{-\infty}^{+\infty}$
\[
\P\big(B\times A^\Z\big)=\mu(B)\quad\text{and}\quad\P\big(B^\Z\times A\big)=\nu(B).
\]
Let $\mathcal J_{\mu,\nu}$ denote the set of couplings of $\mu$ and $\nu$.
The $\bar d$-distance between $\mu$ and $\nu$ is then defined as 
\[
\bar{d}(\mu, \nu) = \inf_{\P\in\mathcal J_{\mu,\nu}} \P\left(\{(\eta,\omega)\in A^\Z\times A^\Z:\eta_0\neq\omega_0\}\right).
\]
It is natural to ask if given two ``close'' (in a sense to be made precise below) probability kernels $g$ 
and $h$, whether we can upper bound the $\bar d$-distance between the respective compatible measures $\mu$ and $\nu$. The 
following result gives such a bound. 


\begin{theo}\label{theo:dbar}
Let $\mu$ be a shift-invariant measure compatible with a kernel $g$ such that
\[
\inf_{a \in A,x\in X^-}g(a|x) > 0
\]
and satisfying either the conditions of Theorem  
\ref{coro:dobrushinconcentration} or of Theorem \ref{coro:summableconcentration}. Let also $\nu$ be  a shift-invariant measure 
compatible with a kernel $h$ with $\lim_j \Var_j(h) = 0$ and $\inf_{a \in A,x\in \X^-}h(a|x) > 0$. The $\bar d$-distance between $
\mu$ and $\nu$ is bounded by 
\begin{equation} \label{eq:dbar1}
\bar d(\mu,\nu)\le \frac{1}{\mathcal C\sqrt{2}}\sqrt{\mathbb{E}_{\nu} \left[\log\frac{h}{g} \right]}
\end{equation}
where $\mathcal C$ equals $\Delta (g)$ or $\Gamma (g)$ depending on which condition $g$ satisfies. 
\end{theo}

\begin{rem}\label{rem:inf}
The conditions 
\[
\inf_{a \in A,x\in X^-}g(a|x) > 0, \inf_{a \in A,x\in X^-}h(a|x) > 0
\]
imply that the alphabet is finite. Although this condition can be weakened, it is the simplest way to guarantee that
$|\log\frac{h}{g}| < \infty$, so that the upper bound on the right hand side of \eqref{eq:dbar1} remains meaningful.
\end{rem}

In Theorem \ref{theo:dbar} we measure the closeness of $h$ and $g$ by $\mathbb{E}_{\nu}\left[\log \frac{h}{g}\right]$. 
It was proved in \cite{coelho_quas_1998}, but without obtaining an explicit upper bound, that if the variation of kernel $g$ satisfies
$\sum_{n\ge r}\prod_{j=r}^{n}(1-(|A|/2)\var_j (g))=\infty$, for some $r\geq 1$, then a small $\|g-h\|_\infty$ implies a small
$\bar d(\mu, \nu)$. As a consequence of Theorem \ref{theo:dbar} we have
\begin{align*}
\mathbb{E}_{\nu}\left[ \log \frac{h}{g}\right]
&=\int\log \frac{h}{g}\, \dd\nu=\int\log\left(1+\frac{h-g}{g}\right)\dd\nu\le\int \frac{h-g}{g}\dd\nu\\
&= \int\sum_{a\in A} h(a|x)\, \frac{h(a|x)-g(a|x)}{g(a|x)}\dd\nu(x)\\
& =\int\sum_{a\in A}\frac{(h(a|x)-g(a|x))^2}{g(a|x)}\dd\nu(x)\\
&\le\frac{1}{\inf g}\int\sum_{a\in A} (h(a|x)-g(a|x))^2\dd\nu(x)\\
& \le \frac{1}{\inf g}\bigg(\sup_{x\in \X^-} \sum_{a \in A}\big|h(a|x)-g(a|x)\big|\bigg)^2
\end{align*}
where $\inf g:=\inf_{a \in A,x\in X^-}g(a|x)$.
Therefore, Theorem \ref{theo:dbar} yields
\begin{equation}\label{eq:dbar}
\bar d(\mu,\nu)\le \frac{1}{\mathcal C\sqrt{2\inf g}}\, \sup_{x\in \X^-} \sum_{a \in A}|h(a|x)-g(a|x)|.
\end{equation}

Theorem \ref{theo:dbar} can also be used to upper bound  the $\bar d$-distance between a measure $\mu$ with kernel $g$ and 
a $k$-step Markov approximation of $\mu$, $\mu^{[k]},k\ge1$. We introduce the kernels
\[
g^{[k]}\big(a|x_{-k}^{-1}\big):=g\big(a|x_{-k}^{-1}y_{-\infty}^{-k-1}\big),\; k\geq 1
\]
for some fixed $y \in \X^-$.
\cite{bressaud_fernandez_galves_1999b,fernandez/galves/2002} showed that if the kernel $g$ satisfies $\Gamma (g)>0$, then 
there exists a  sequence $\mu^{[k]},k\ge1$ such that $\bar d(\mu,\mu^{[k]})\le C\var_k (g)$ where $C$ is some positive constant. 
Later  \cite{gallo/lerasle/takahashi/2013} extended this result, obtaining, via ``coupling from the past'' arguments,   upper bounds 
in the case where $\sum_{n\ge1}\prod_{k=1}^{n}(1-\var_k (g))=\infty$, but their results are not explicit, depending on the tail 
distribution of the time for success in the coupling. Such bounds  have proved to be a valuable tool to 
obtain further properties of the measure $\mu$ \cite[for instance]{collet/duarte/galves/2005}.  Here,  if 
either $\Gamma (g)>0$ or $\Delta (g)>0$, using  \eqref{eq:dbar} we obtain the following result (where we use $\var_k(g)$ since $A$ is finite).
\begin{coro} \label{coro:dbar}
For all $k\geq 1$ we have  
\[
\bar d(\mu,\mu^{[k]})\le \frac{|A|}{2\sqrt{2}\,\mathcal{C}\sqrt{\inf g}}\,\var_k (g)
\]
where $\mathcal C$ equals $\Delta (g)$ or $\Gamma (g)$ depending on which condition $g$ satisfies. 
\end{coro}
\begin{proof}
Apply \eqref{eq:dbar} with $\nu=\mu^{[k]}$.
\end{proof}
\begin{rem}
A similar bound was obtained by \cite{fernandez/galves/2002} under the assumption that $\Gamma (g)>0$, while Corollary 
\ref{coro:dbar} also  holds if $\Delta (g)>0$. We refer the reader to Subsection \ref{remarksonmainresults} where an example 
satisfying $\Delta (g)>0$, but such that $\Gamma (g)=0$, is provided, showing that our result is strictly more general than the one 
in \cite{fernandez/galves/2002}. 
\end{rem}

\begin{ex}[Bramson-Kalikow-Friedli model]
Let $A=\{-1,+1\}$, $\varepsilon \in(0,1/2)$, $(\lambda_j)_{j \geq 1}$ be a sequence of positive real numbers such that
$\sum_{j = 1}^\infty \lambda_j = 1$, and $(m_j)_{j \geq 1}$ be an increasing sequence of positive odd integers.
Let also $\varphi:[-1,1]\rightarrow[\varepsilon,1-\varepsilon]$ be a monotonically increasing function satisfying $\varphi(-s)+\varphi(s)=1$ for $s\in[-1,1]$. 
The Bramson-Kalikow-Friedli model is given by
\[
g(+1|x) = \sum_{j=1}^\infty \lambda_j\,\varphi\left(\frac{1}{m_j}\sum_{i=1}^{m_j}x_{-i}\right)\,,\,\,x\in\X^-.
\] 
There always exists at least one compatible measure $\mu$ since the alphabet is finite and $\Var_j(g)$ vanishes in $j$.
If $\varphi(s)=\varepsilon+(1-2\varepsilon)\,\mathds{1}_{\Z_{<0}}(s)$, we get the original model introduced by Bramson and Kalikow
\cite{bramson/kalikow/1993}. They showed that the sequences $(\lambda_j)_{j \geq 1}$ and $(m_j)_{j \geq 1}$ can be chosen so that the 
corresponding kernel exhibits multiple compatible shift-invariant measures \cite{bramson/kalikow/1993}. 

For all $k \geq 1$, the $m_k$-step Markov approximation $\mu^{[m_k]}$ is defined by 
the kernel
\begin{align*}
g^{[m_k]}(+1|x)
= \sum_{j=1}^k \lambda_j\, \varphi\left(\frac{1}{m_j}\sum_{i=1}^{m_j}x_{-i}\right) + (1-\varepsilon)\sum_{j > k}\lambda_j.
\end{align*}
The sequence of measures $\mu^{[m_k]},k\ge1$ was used by \cite{gallesco/gallo/takahashi/2014} for their proof of phase transition
of the Bramson-Kalikow model. 

Suppose for now that there exists a shift-invariant  measure $\mu$ compatible with $g$. We can easily derive a lower bound for
$\bar{d}(\mu,\mu^{[m_k]})$. Indeed, from the definition of $\bar{d}$-distance, we have that 
\[
\bar{d}\big(\mu,\mu^{[m_k]}\big) \geq |\mu^{[m_k]}([1]) -\mu([1])|.
\]
By symmetry and uniqueness of $\mu$, we have that $\mu([1]) = 1/2$.
A direct calculation then shows 
that 
\[
\mu^{[m_k]}([1]) \geq \varepsilon\sum_{j > k} \lambda_j + 1/2,
\]
hence
\[ 
 \bar{d}\big(\mu,\mu^{[m_k]}\big) \geq \varepsilon\sum_{j > k} \lambda_j\,.
\]

If $\Gamma(g)>0$ or $\Delta(g)>0$, Corollary \ref{coro:dbar} allows us to show that this bound is actually of the right order in $k$ since we have that for all $k \geq 1$,
\[
\bar d\big(\mu,\mu^{[m_k]}\big)\le \frac{1}{2\sqrt{2\varepsilon}\,\Gamma(g)}\sum_{j > k} \lambda_j.
\]
Thus, it only remains to give examples of kernels for which $\Gamma(g)>0$ or $\Delta(g)>0$. First, observe that, for any function
$\varphi$, we have
\[
\sum_{j\ge1} \Var_j (g) \leq \sum_{j \ge1} m_j \sum_{i \geq j} \lambda_i.
\]
So if $\sum_{j = 1}^\infty m_j \sum_{i \geq j} \lambda_i<\infty$, we have $\Gamma(g)>0$ and $\Var_0(g)=1-2\varepsilon>0$, 
independently of the function $\varphi$. Observe that, if $\sum_{j = 1}^\infty m_j \sum_{i \geq j} \lambda_i=\infty$, we can still have 
examples in which $\Delta(g)>0$ and use Corollary \ref{coro:dbar}. For example,  take $
\varphi(s):=\frac{1}{2}+\left(\frac{1}{2}-\epsilon\right)s$, which was studied in \cite{friedli2015note}. In this case, a simple calculation 
shows that 
\[
\sum_{j\ge1}\Osc_j( g)\le (1-2\varepsilon)\sum_{j\ge1}\lambda_j<1
\]
and therefore $\Delta(\bar g)>0$, independently of the choice of the sequences $(\lambda_j)_{j \geq 1}$ and $(m_j)_{j \geq 1}$.
\end{ex}

\subsection{Concentration of functions that depend on infinitely many coordinates}\label{sec:app_infinite}

A natural application of Theorem \ref{concentrationformoregeneralfunctions} is the following. For $n\geq 1$,
let $S_n \phi:=\phi+\phi\circ T+\cdots+\phi\circ 
T^{n-1}$ where $\phi\in \mathrm{UC}(A^{\N})\cap \mathrm{BC}(A^{\N})$ and satisfies $\|\ushort{\delta} (\phi)\|_1<+\infty$. Then we 
have, for all $n\geq 1$ and $u>0$,
\begin{equation}\label{avion}
\mu\left(\left |\frac{S_n \phi}{n}  -\int \phi \dd\mu \right| > u\right)\leq 
2\exp \left( -\frac{2n u^2}{\mathcal C\|\ushort{\delta}(\phi)\|_1^2}\right)
\end{equation}
where $\mathcal C$ equals $\Delta (g)$ or $\Gamma (g)$ depending on which condition $g$ satisfies. The proof is as follows. 
Taking $f=S_n \phi$, we can check that 
\begin{equation}\label{Young}
\|\ushort{\delta}(S_n) \phi\|_2^2\leq n \|\ushort{\delta}(\phi)\|_1^2, \, n\geq 1.
\end{equation}
Then we apply \eqref{Chernoff-GCB} to get \eqref{avion}. To prove \eqref{Young}, observe that
$\delta_j(S_n) \phi\leq \sum_{i=0}^{n-1} \delta_{i+j}(\phi)$, and apply Young's inequality for (discrete) convolutions: if
$\ushort{v} \in \ell^p(\N)$ and $\ushort{w} \in \ell^q(\N)$, for some $1\leq p\leq q\leq +\infty$, then
$\ushort{v} * \ushort{w}\in \ell^{r}(\N)$ where $r\geq 1$ satisfies $1+r^{-1}=p^{-1}+q^{-1}$, and
$\|\ushort{v} * \ushort{w}\|_r \leq \|\ushort{v}\|_p \|\ushort{w}\|_q$. We use it with $r=2, p=2, q=1$,
$v_k=\mathds{1}_{\llbracket 0,n-1\rrbracket}(k)$ and $w_k=\delta_k(h)$.
(\footnote{Note that we don't have a convolution defined as usual, but one can readily check that the proof of Young's inequality 
works if we use $\sum_{i\geq 0} u_i v_{j+i}$ instead of $\sum_{i\geq 0} u_i v_{j-i}$.})

\section{Proofs of the results}\label{sec:proofs}

\subsection{Gaussian concentration bound using coupling}

All the GCBs obtained in this work are consequences of an abstract GCB proved in \cite{chazottes/collet/kulske/regig/2007} for 
finite alphabet processes. 
The point is then to have a good control on a certain ``coupling matrix'', which is what we do hereafter for stochastic chains with 
unbounded memory. 
We will state it in a form more adapted for our purpose. 

Initially we write  $f(\sigma_{0}^n)-\E[f(\sigma_{0}^n)]$ as a sum of martingale differences. Defining
$V_k(\sigma) := \E_{\mu}[f|\F_k](\sigma)-\E_{\mu}[f|\F_{k-1}](\sigma)$, we have
\[
f(\sigma_{0}^n)-\E_{\mu}[f(\sigma_{0}^n)] = \sum_{k = 0}^n\big(\E_{\mu}[f|\F_k](\sigma)-\E_{\mu}[f|\F_{k-1}](\sigma)\big) = 
\sum_{k=0}^{n} V_k(\sigma).
\]
Observe that $\E_{\mu}[f|\F_k] = \sum_{\omega_{k+1}^n}f(\sigma_{0}^k\omega_{k+1}^n)\mu([\omega_{k+1}^n]|\sigma_{0}^k)$, 
where $\mu(B|\sigma_{0}^k):=\mu(B\cap[\sigma_{0}^k])/\mu([\sigma_{0}^k])$ for all measurable sets $B$. Now, we will obtain an 
upper bound on $V_k(\sigma)$ based on coupling.

\begin{lemma} \label{lemma:martingalebound}
For $\sigma \in A^{\llbracket 0, \infty\rrbracket}$, $a,b \in A$, $n \geq 1$, and $j \geq 0$, let $\nu^{\sigma, a,b}_j$ be any coupling 
between $\mu(\cdot|\sigma_0^{j-1}a)$ and $\mu(\cdot|\sigma_0^{j-1}b)$. For all $k \in \llbracket 0, n\rrbracket$ we have
\[
V_k(\sigma) \leq \delta_k(f) + \sup_{a,b\in A}\sum_{j = 1}^{n-k-1} \nu^{\sigma, a,b}_k( \eta_{k+j} \neq \omega_{k+j})\, \delta_{k+j}(f).
\]
\end{lemma}

\begin{proof}
Following \cite{chazottes/collet/kulske/regig/2007} we have for $\sigma \in A^{n+1}$
We have
\begin{align}
V_{k}(\sigma) &=  \sum_{\omega_{k+1}^n}
f(\sigma_{0}^k\omega_{k+1}^n)\,\mu([\omega_{k+1}^n]|\sigma_{0}^k) -  \sum_{\omega_{k}^n}f(\sigma_{0}^{k-1}\omega_{k}^n)
\,\mu\big([\omega_{k}^n]|\sigma_{0}^{k-1}\big) \notag \\
&= \sum_{\omega_{k+1}^n}f(\sigma_{0}^k\omega_{k+1}^n)\,\mu\big([\omega_{k+1}^n]|\sigma_{0}^k\big) \\
&\quad- \sum_{\omega_{k}^n}f(\sigma_{0}^{k-1}\omega_{k}^n)\,\mu\big([\omega_{k+1}^n]|\sigma_{0}^{k-1}\omega_k\big)
\,\mu\big([\omega_k]|\sigma_{0}^{k-1}\big) \notag\\
& \leq \sup_{a\in A} \sum_{\omega_{k+1}^n}f(\sigma_{0}^{k-1}a\,\omega_{k+1}^n)\,\mu\big([\omega_{k+1}^n]|\sigma_{1}^{k-1}a\big) \notag \\
&\quad - \inf_{b \in A}\sum_{\omega_{k}^n}f(\sigma_{0}^{k-1}b\,\omega_{k+1}^n)\,\mu\big([\omega_{k+1}^n]|\sigma_{1}^{k-1}b\big)
\,\mu\big([\omega_k]|\sigma_{0}^{k-1}\big) \notag\\
&\leq \sup_{a\in A} \sum_{\omega_{k+1}^n}f(\sigma_{0}^{k-1}a\,\omega_{k+1}^n)\,\mu\big([\omega_{k+1}^n]|\sigma_{0}^{k-1}a\big) \notag\\
& \quad - \inf_{b \in A}\sum_{\omega_{k+1}^n}f(\sigma_{0}^{k-1}b\,\omega_{k+1}^n)\,\mu\big([\omega_{k+1}^n]|\sigma_{0}^{k-1}b\big). 
\label{eq:beforecouple}
\end{align}
Let $\eta_k := a$ and $\omega_k := b$. We have 
\begin{align}
\notag
& |f(\sigma_{1}^{k-1}a\,\eta_{k+1}^n) - f(\sigma_{1}^{k-1}b\,\omega_{k+1}^n)|\\
\notag
&\leq \sum_{j = 0}^{n-k} |f(\sigma_{1}^{k-1}\omega_{k}^{k-1+j}\eta_{k+j}^n) - f(\sigma_{1}^{k-1}\omega_{k}^{k+j}\eta_{k+j+1}^n) | 
\\
&\leq \sum_{j = 0}^{n-k} \delta_{k+j}(f) \, \mathds{1}_{\{\eta_{k+j} \neq \omega_{k+j}\}}. 
\label{eq:oscilating}
\end{align}

Hence, from \eqref{eq:beforecouple} and  \eqref{eq:oscilating}, we have
\begin{align*}
V_{k}(\sigma) &\leq \sup_{a,b\in A}
\sum_{\omega_{k+1}^n}
|f(\sigma_{0}^{k-1}a\,\eta_{k+1}^n) - f(\sigma_{0}^{k-1}b\,\omega_{k+1}^n)|\; \nu^{\sigma, a,b}_k\big([\eta_{k+1}^n,\omega_{k+1}^n]\big) 
\notag \\
&\leq \sup_{a,b\in A} \sum_{\eta_{k+1}^n, \omega_{k+1}^n}
\sum_{j = 0}^{n-k} \delta_{k+j}(f) \,\mathds{1}_{\{\eta_{k+j} \neq \omega_{k+j}\}}\, \nu^{\sigma, a,b}_k\big([\eta_{k+1}^n,\omega_{k+1}^n]\big)
\notag\\
&\leq \delta_k(f) + \sup_{a,b\in A}\sum_{j = 1}^{n-k} \nu^{\sigma, a,b}_k( \eta_{k+j} \neq \omega_{k+j}) \, \delta_{k+j}(f).
\end{align*}
This ends the proof of the lemma.
\end{proof}
At the end of the proof we used the notation $[a_m^n,b_m^n]:=\{(\omega,\eta)\in A^\Z\times A^\Z:\omega_m^n=a_m^n,
\eta_m^n=b_m^n\}$ as a natural extension for denoting cylinder sets in $A^\Z\times A^\Z$.

\begin{theo} \label{theo:basicGCB}
For $\sigma \in A^{\llbracket 0, \infty\rrbracket}$, $a,b \in A$, $n \geq 1$, and $k \geq 1$, let $\nu^{\sigma, a,b}_k$ be any coupling 
between $\mu\big(\cdot|\sigma_0^{k-1}a\big)$ and $\mu\big(\cdot|\sigma_0^{k-1}b\big)$. Also, define
\[
r = \sum_{j = 1}^{\infty}\sup_k\sup_{\sigma}\sup_{a,b}\nu^{\sigma, a,b}_{k}\big(\eta_{k+j} \neq \omega_{k+j}\big). 
\]
For all $\theta\in\R$, $n \geq 1$ and $f:A^n \to \R$ such that $\delta_j(f)<+\infty$ for $j=0,\ldots,n$, we have 
\[
\E_{\mu}\left[\e^{\theta (f  -\E_{\mu}f)}\right] 
\leq \exp\bigg(\frac{\theta^2 (1+{r})^2}{8}\,  \|\ushort{\delta}(f)\|^2_2\bigg).
\]
As a consequence, we get, for all $u>0$,
\[
\mu(|f  - \E_{\mu}[f]| > u) \leq 2 \exp\left(-\frac{2u^2}{(1+{r})^2\, \|\ushort\delta(f)\|_2^2}\right).
\]
\end{theo}
\begin{proof}
Define
\[
U_k(\sigma) = 
\sup_{a\in A} \sum_{\omega_{k+ 1}^n}f(\sigma_{0}^{k-1}a\,\omega_{k+1 }^n)\,\mu\big([\omega_{k+1}^n]|\sigma_{0}^{k-1}a\big)
-\E_{\mu}[f|\F_{k-1}](\sigma)
\]
and
\[
L_k(\sigma) = \inf_{b\in A} \sum_{\omega_{k+1}^n}f(\sigma_{0}^{k-1}b\,\omega_{k+1}^n)\,\mu\big([\omega_{k+1}^n]|\sigma_{0}^{k-1}b\big)
-\E_{\mu}[f|\F_{k-1}](\sigma).
\]
For $k,j \geq 0$, let us also define
\[
D_{k,k+j} := \sup_\sigma\sup_{a,b\,\in A}\nu^{\sigma, a,b}_k( \eta_{k+j} \neq \omega_{k+j}).  
\]
From Lemma \ref{lemma:martingalebound}, we have
\[
U_k -L_k
\leq \delta_k(f) + \sum_{j = 1}^{n-k} \sup_{\sigma}\sup_{a,b\,\in A}\nu^{\sigma, a,b}_k( \eta_{k+j} \neq \omega_{k+j})\, \delta_{k+j}(f)
= \sum_{j = 0}^{n-k}D_{k,k+j}\, \delta_{k+j}(f)\,.
\]
Now observe that $L_k\le V_k\le L_k+(U_k-L_k)$, and thus, using Lemma 2.3 of \cite{devroye2012combinatorial} and then 
proceeding as in the proof of Theorem 1 of \cite{chazottes/collet/kulske/regig/2007} we get for all $\theta>0$
\[
\E_{\mu}\left[\e^{\theta (f  -\E_{\mu}[f])}\right] 
\leq \exp\left(\frac{\theta^2 \|D\|_2^2\,  \|\ushort{\delta}(f)\|^2_2}{8}\right)
\]
and, for all $u>0$,  
\[
\mu\left(|f  -\E_{\mu}[f]| \geq u\right) \leq 2\exp\left(-\frac{2u^2}{\|D\|_2^2 \, \|\ushort{\delta}(f)\|_{2}^2}\right). 
\]
Using the inequality $\|D\|_2^2 \leq \|D\|_1\|D\|_\infty \leq (1+r)^2$, we conclude the proof of the theorem.
\end{proof}

\subsection{One-step maximal coupling}

Here we introduce what is called the one-step maximal coupling. We will assume without loss 
of generality that $A = \{1,\ldots, |A|\}$, when $A$ is finite, and $A = \N$ when $A$ is infinite. 
We will define a probability kernel on $p:A\times A \times \X^- \times \X^-  \to [0,1]$ as follows. For $(s,s') \in A \times A$ and $(x,x') 
\in \X^- \times \X^-$, we put
\[
\sum_{c\geq s}\sum_{d\geq s'}p(c,d| x,x') = \sum_{c \geq s} g(c|x) \wedge \sum_{d \geq s'} g(d|x').
\]
Let us denote by $\P$ the measure specified by the kernel $p$. The following equalities
\[
\sum_{c\geq s}\sum_{d\geq 1}p(c,d| x,x') = \sum_{c \geq s} g(c|x)\; \;\; \text{and}\;\;\;\sum_{c\geq 1}\sum_{d\geq s'}p(c,d| x,x') = 
\sum_{d \geq s'} g(d|x')
\]
imply that $\P$ is a coupling of two copies of the process specified by $g$. It is called one-step maximal coupling because it 
maximizes the probability of agreement (diagonal of the coupling) at each step, given any pair of pasts,
\[
p(s,s| x,x') =  g(s|x) \wedge g(s|x').
\]
In particular, notice that
\[
\sum_{c \neq d}p(c,d| x,x') = 1 - \sum_{s \in A} g(s|x) \wedge g(s|x') = \frac{1}{2}\sum_{s \in A} |\, g(s|x) - g(s|x')|.
\]

\subsection{Bounding the coupling error by oscillation}

We have the following important lemma.
\begin{lemma} \label{lemma:dobrushin}
Take any $x \in \X^-$. For all $a,b\in A$, let $\P^{x,a,b}$ be the one-step maximal coupling between $P^{xa}$ and $P^{xb}$. 
Then, for all $j \geq 1$, we have
 \[ 
\P^{x,a,b}\left(\eta_j \neq \omega_j \right) \leq \Osc_j (g)+\sum_{k = 1}^{j-1}\Osc_{j-k} (g)
\,\P^{x,a,b}\left( \eta_k \neq \omega_k  \right).
 \]
\end{lemma}
\begin{proof}

We want to compute $\P^{x,a,b}(\eta_i \neq \omega_i)$, which equals
\begin{align*}
\sum_{y_1^{i-1},z_1^{i-1}}\P^{x,a,b}([y_1^{i-1},z_1^{i-1}])\,\P^{x,a,b}\big(\eta_i \neq \omega_i|[y_1^{i-1},z_1^{i-1}]\big).
\end{align*}
Under the maximal coupling we have
\[
\P^{x,a,b}\big(\eta_i \neq \omega_i|[y_1^{i-1},z_1^{i-1}]\big) \leq
\frac{1}{2}\sum_{s \in A} \big|g\big(s|xay_1^{i-1}\big)-g\big(s|xbz_1^{i-1}\big)\big|.
\]
Let $y_0 := a$ and $z_0 := b$. We get for each $s\in A$
\begin{align*}
\big|g(s|xa y_1^{i-1})-g(s|xbz_1^{i-1})\big|
&= \Big|\sum_{j=0}^{i}g\big(s|xz_{0}^{j-1}y_j^{i-1}\big)-g\big(s|xz_{0}^{j}y_{j+1}^{i-1}\big) \Big|\\
&\leq\sum_{j=0}^{i-1} \big|\, g\big(s|xz_{0}^{j-1}y_j^{i-1}\big)-g\big(s|xz_{0}^{j}y_{j+1}^{i-1}\big)\big|.
\end{align*}
Putting $y_0 := a$ and $z_0 := b$, we have the bound
\[
\frac{1}{2}\sum_{s \in A} \big|g\big(s|xay_1^{i-1}\big)-g\big(s|xbz_1^{i-1}\big)\big| \le \sum_{j=0}^{i-1}\mathds{1}_{\{y_{j} \neq z_j\}}\Osc_{i-j}(g).
\]
Hence
\begin{align*}
\P^{x,a,b}(\eta_i \neq \omega_i)
& \le  \sum_{y_1^{i-1},z_1^{i-1}}\P^{x,a,b}\big([y_1^{i-1}, z_1^{i-1}]\big)\sum_{j=0}^{i-1}\mathds{1}_{\{y_{j} \neq 
z_j\}}\Osc_{i-j}(g)\\
& = \sum_{j=0}^{i-1}\sum_{y_1^{i-1},z_1^{i-1}}\P^{x,a,b}\big([y_1^{i-1}, z_1^{i-1}]\big)\mathds{1}_{\{y_{j} \neq z_j\}}\Osc_{i-j}(g) \\
&= \sum_{j=0}^{i-1}\sum_{y_{j}\neq z_{j}}\P^{x,a,b}\big([y_{j}, z_{j}]\big)\Osc_{i-j}(g)\\&\leq  \Osc_i (g)
+ \sum_{j=1}^{i-1}\P^{x,a,b}(\eta_{j}\ne \omega_{j})\Osc_{i-j} (g)
\end{align*}
which concludes the proof.
\end{proof}

The following result is a straightforward consequence of Lemma \ref{lemma:dobrushin}.
\begin{prop} \label{prop:dobrushin}
For all $a,b\in A$ and $x \in \X^-$, let $\P^{x,a,b}$ be the one-step maximal coupling between $P^{xa}$ and $P^{xb}$. 
If $\Delta (g) > 0$, then, for all $n \geq 1$, we have
\[
\sum_{j = 1}^n \sup_{\substack{ a,b \in A\\x \in \X^- }}\P^{x,a,b}(\eta_j \neq \omega_j) \leq \frac{1-\Delta (g)}{\Delta (g)}. 
\]
\end{prop}

\begin{proof}
From Lemma \ref{lemma:dobrushin}, we have
\begin{equation}  \label{eq:supdobrushin}
\sup_{\substack{a,b \in A\\x \in \X^- }}\P^{x,a,b}\left(\eta_j \neq \omega_j \right) \leq \Osc_j (g)+\sum_{k = 1}^{j-1}\Osc_{j-k}(g)
\sup_{\substack{a,b \in A\\x \in \X^- }}\P^{x,a,b}\left( \eta_k \neq \omega_k  \right).
\end{equation}

Define vectors $\alpha$ and $\beta$ such that for $i \geq 1$, $\alpha_i = \sup_{\substack{a,b \in A\\x \in \X^- }} \P^{x,a,b}\left(\eta_i 
\neq \omega_i \right)$ and $\beta_i = \Osc_i (g)$. We also define a matrix $L$ such that for $i > j$, $L_{ij} = \Osc_{i-j} (g) $ and 
$L_{ij} = 0$ otherwise. Hence, from \eqref{eq:supdobrushin} we have
\[
\alpha \leq L\alpha + \beta.
\]
Therefore, 
\[
\|\alpha\|_{1} \leq \frac{\|\beta\|_{1}}{1-\|L\|_1}  = \frac{\sum_{j=1}^\infty \Osc_j (g)}{1-\sum_{j=1}^\infty \Osc_j (g)}
\]
as we wanted to prove.
\end{proof}

\subsection{Proofs of Theorems \ref{theo:dobrushinconcentration} and  \ref{coro:dobrushinconcentration}}
\label{subsec:proofcoro:dobrushinconcentration}

\begin{proof}[Proof of Theorem \ref{theo:dobrushinconcentration}]
In the statement of Theorem \ref{theo:basicGCB}, for $\sigma \in A^{\llbracket 0, \infty\rrbracket}$, $a,b \in A$, $n \geq 1$, and $k 
\geq 1$, let $\nu^{\sigma, a,b}_k$ be the one-step maximal coupling between $P^{x\sigma_0^{k-1}a} = P^x(\cdot | \sigma_0^{k-1}a)
$ and $P^{x\sigma_0^{k-1}b} = P^x(\cdot | \sigma_0^{k-1}b)$. In this case, for all $x \in \X^-$, Proposition \ref{prop:dobrushin} 
implies that $1+r \leq \Delta (g)^{-1}$, proving Theorem \ref{theo:dobrushinconcentration}. 
\end{proof}

In order to prove Theorem \ref{coro:dobrushinconcentration}, we need the following lemma. Recall that  $T$ denotes the shift 
operator defined by $(T x)_i = x_{i+1}$ and define the shifted probability measures  $P^x_j = P^x \circ T^{-j},j\ge1$ in which 
$P^x$ is the measure compatible with $g$ started from the fixed past $x\in \X^-$. We need the following definitions and facts on 
weak convergence of probability measures. A sequence of probability measures  $(\mu_n)_n$ on 
$A^{\Z}$ weakly converges to $\mu$ if for every bounded continuous function $h:A^{\Z}\to\R$ we have
$\int h \dd\mu_n\to \int h\dd\mu$. By Remark 3 of Section 4.1 in \cite{georgii2011gibbs}, this is equivalent to the convergence
$\mu_n(C)\to\mu(C)$ for every cylinder $C$. (This is well known when $A$ is finite.)

\begin{lemma}\label{lem:thermo_limit}
Under the assumptions of Theorem \ref{coro:dobrushinconcentration}, $(P^{T^{-j}x}_j)_{j\ge1}$ 
converges weakly to $\mu$ for $\mu$-a.e $x$. 
\end{lemma}
\begin{proof}
We need to find a set $\mathcal{S}\subset\mathcal{X}^-$ satisfying the following conditions: $\mu(\mathcal{S})=1$, and for any 
$x\in\mathcal{S}$, $P^{T^{-j}x}_j(C)$ converges to $\mu(C)$ for all cylinders $C$.
Fix a cylinder set $C$. According to Definitions \ref{def:fixedpastproba} and \ref{def:compa}, we have, for sufficiently large $j$'s, 
$P^{T^{-j}x}_j(C)=\mu(C|\mathcal{F}_{-\infty}^{-j})(x)$ for $\mu$-almost every past $x$. On the other hand, 
the Reverse Martingale Theorem ensures that $\mu(C|\mathcal{F}_{-\infty}^{-j})(x)\rightarrow\mu(C|\mathcal{F}_{-\infty})(x)$
for $\mu$-almost every past $x$, where $\mathcal{F}_{-\infty}:=\bigcap_{j\ge1}\mathcal{F}_{-\infty}^{-j}$ is 
the left tail sigma-algebra.  By \cite[Theorem 4.6]{fernandez/maillard/2005}, if $g$ is continuous and satisfies $\Delta(g)>0$, then 
the measure $\mu$ is the unique measure compatible with $g$. We are in force of both assumptions since on a countable 
alphabet, condition $\Delta(g)>0$ automatically implies continuity. Thus $\mu$ is the unique measure compatible with $g$. Now, 
invoking \cite[Theorem 3.2 items (a) and (b)]{fernandez/maillard/2005}, we get that $\mu$  is trivial on $\mathcal{F}_{-\infty}$.
It follows that $\mu(C|\mathcal{F}_{-\infty})(x)=\mu(C)$ for $\mu$-almost every past $x$. We have therefore proved that there exists 
a set $\mathcal{S}(C)\subset\X^-$ such that $\mu(\mathcal{S}(C))=1$ and $P^{T^{-j}x}_j(C)$ converges to $\mu(C)$ for any 
$x\in\mathcal{S}(C)$. We conclude observing that, since the alphabet is countable, there are countably many cylinders and we can 
take $\mathcal{S}=\bigcap_C\mathcal{S}(C)$. 
\end{proof}

We are now ready for the proof of Theorem \ref{coro:dobrushinconcentration}.
\begin{proof}[Proof of Theorem \ref{coro:dobrushinconcentration}]
Let $f\in\mathcal{L}$. Recall that $f$ is bounded and continuous.
Lemma \ref{lem:thermo_limit} above guarantees that $P_j^{T^{-j}x}$ converges weakly to $\mu$, for $\mu$-a.e. $x$. 
So let us fix such a point $x$ and take $\theta\in\R_+$. For $\sigma \in A^{\llbracket 0, n+j \rrbracket}$, let
$f_j(\sigma) := f(\sigma_j^{j+n+1})$.
Now, for any $\epsilon > 0$, there is a $j_0$ such that for all $j \geq j_0$ we have
\begin{align*}
\E_{\mu}\left[\exp\Big(\theta \big(f  -\E_{\mu}[f]\big)\Big)\right]
& \leq \left(\E_{P^{T^{-j}x}_j}\left[\exp\Big(\theta \big(f  -\E_{P^{T^{-j}x}_j}[f]\big)\Big)\right]+\epsilon\right)\exp(\theta \epsilon) \\
& \leq \left(\E_{P^{T^{-j}x}}\left[\exp\Big(\theta \big(f_j  -\E_{P^{T^{-j}x}}[f_j]\big)\Big)\right]+\epsilon\right)\exp(\theta \epsilon) \\
& \leq \left(\exp\Big(\frac{\theta^2 \Delta(g)^{-2}}{8} \|\ushort{\delta}(f)\|^2_2\Big)+\epsilon\right)\exp(\theta \epsilon)
\end{align*}
where the last inequality follows from Theorem \ref{theo:dobrushinconcentration} because
$f_j\in \mathcal{L}$ and $\|\ushort\delta(f)\|_2 = \|\ushort\delta (f_j)\|_2$. Taking $\epsilon \rightarrow 0$, we conclude the proof
for $\theta\in\R_+$. The proof for $\theta\in\R_-$ is very similar.
\end{proof}

\subsection{Bounding the coupling error by variation}

Fix $y, z \in A^{{\llbracket -\infty, 0 \rrbracket}}$. Let $\P^{y,z}$ be the one-step maximal coupling between measures $P^{y}$ and 
$P^{z}$ with kernel $g$. We want to obtain an upper bound $\P^{y,z}\left(\eta_j \neq \omega_j \right)$. To achieve this, we will use an auxiliary  process. 
Given $x \in \X^-$, let $\ell(x) = \inf\{k \geq 1: x_{-k} = 1\}$ and $\ell(\ldots00) = \infty$. Consider the
kernel $h$ associated with $g$ as $h^g:\{0,1\}\times\{0,1\}^{\llbracket -\infty, -1\rrbracket} \to (0,1)$ where $h^g(1|x) = q_{\ell(x)}$ 
and $q_j = \Var_j(g)$. For any $x\in \{0,1\}^{\X^-}$ consider the measure $P_h^{x1}$ constructed as in Definition \ref{def:fixedpastproba} using $h$ in 
place of $g$. It is the 
undelayed renewal measure, and will be our auxiliary process. Recall the definition of the projection functions $\pi_{i}^j$, $i\le j$ given in Section 
\ref{sec:def}, and put $\pi_i:=\pi_i^i$ as the projection on the single coordinate $i$. 

\begin{lemma}\label{lemma:comparison}
We have that, 
for all $y,z \in A^{{\llbracket -\infty, 0 \rrbracket}}$ and all $j \geq 0$,
\begin{equation*}
\P^{y,z}\left(\eta_j \neq \omega_j \right)  \leq P_h^{x1}(\pi_j = 1).
\end{equation*}
\end{lemma}
\begin{proof}
For any $\eta,\omega$ in $A^\Z$, let $\sigma_j(\eta,\omega): = \mathds{1}_{\{\eta_j \neq \omega_j \}}$ for $j \in \Z$.
By definition, for all $j \geq 0$, we have
\begin{equation*}
\sup_{y,z \in A^{{\llbracket -\infty, 0 \rrbracket}}}\P^{y,z}\left(\sigma_j = 1 | \sigma^{j-1}_1 = 0, \sigma_{0} = 1 \right)  \leq \Var_{j+1}(g) = q_{j+1}.
\end{equation*}
By stochastic domination, we conclude that 
\begin{equation*}
\P^{y,z}\left(\sigma_j = 1 \right)  \leq P_h^{x1}(\pi_j = 1)
\end{equation*}
as we wanted to prove.
\end{proof}

\begin{lemma} \label{lemma:variation}
Let $\gamma_1(g) = \Var_0(g)$ and for $k \geq 2$, $\gamma_k (g):= \Var_{k-1} (g)
\prod_{i=0}^{k-2}(1-\Var_i (g))$. If $P_h^{x1}$ be the measure specified by renewal kernel $h^g$ starting with $x_{0} = 1$ then, 
for all $j \geq 1$, we have the renewal equation
\[ 
P_h^{x1}\left(\pi_j =1 \right) = \gamma_j(g)+\sum_{k = 1}^{j-1}\gamma_{j-k}(g)
\,P_h^{x1}\left( \pi_{k} =1  
\right).
 \]
\end{lemma}
\begin{proof}
We want to compute (for any $i\ge2$ we denote by $0^i$ the string $00\ldots0$ of $i$ consecutive $0$)
\begin{equation*}
P_h^{x1}(\pi_j =1) = P_h^{x1}\big(\pi_j =1, \pi_1^{j-1}=0^{j-1}\big)
+\sum_{k=1}^{j-1}P_h^{x1}\big(\pi_j =1, \pi_{j-k+1}^{j-1}= 0^{k-1}, \pi_{j-k} = 1\big).
\end{equation*}
For $k \geq 1$, we have that
\begin{align*}
P_h^{x1}\big(\pi_j = 1, \pi_{j-k+1}^{j-1} = 0^{k-1}, \pi_{j-k} =1\big) 
& =  P_h^{x1}\big(\pi_j =1 \mid \pi_{j-k+1}^{j-1}=0^{k-1}, \pi_{j-k} =1\big)\\
& \quad\times \prod_{i=j-k+1}^{j-1} P_h^{x1}\big(\pi_i = 0 \mid \pi_{j-k+1}^{i-1}=0^{k-1}, \pi_{j-k} =1\big)\\
&= \Var_{k-1} (g) \prod_{i = 0}^{k-2}(1-\Var_i (g)) \,P_h^{x1}(\pi_{j-k} =1)\\
& = \gamma_k (g) \, P_h^{x1}(\pi_{j-k} =1)
\end{align*}
where we used the convention $\prod_{i=0}^{-1}=1$. Similarly,
\[
P_h^{x1}\big(\pi_j =1, \pi_1^{j-1}=0^{j-1}\big) = \gamma_j (g).
\]
Therefore, we have
\[
P_h^{x1}(\pi_j =1) \leq \gamma_j (g)+\sum_{k=1}^{j-1}\gamma_k (g) \,P_h^{x1}(\pi_{j-k} = 1). 
\]
Using the symmetry between the indices $k$ and $j-k$ in the summation, we conclude the proof.
\end{proof}

The next result is a direct consequence of  Lemmas \ref{lemma:comparison} and \ref{lemma:variation}.
\begin{prop} \label{prop:boundingbyvariation}
For all $y, z \in A^{{\llbracket -\infty, 0 \rrbracket}}$, let $\P^{y,z}$ be the one-step maximal coupling between $P^{y}$ and $P^{z}$. If 
$\Gamma (g):= \prod_{j = 0}^\infty (1-\Var_j (g)) > 0$, then, for all $n\geq 1$, we have
\[
\sum_{j = 1}^n\, \sup_{y,z\in A^{{\llbracket -\infty, 0 \rrbracket}}}\P^{y,z}(\eta_j \neq \omega_j) \leq \frac{1-\Gamma (g)}{\Gamma (g)}. 
\]
\end{prop}

\begin{proof}
From Lemma \ref{lemma:variation}, we have
\begin{equation}  \label{eq:supvariation}
P_h^{x1}\left(\pi_j =1 \right) = \gamma_j (g)+\sum_{k = 1}^{j-1}\gamma_{j-k} (g)P_h^{x1}\left( \pi_k =1  \right).
\end{equation}
Define vector $\alpha$ such that for $j \geq 1$, $\alpha_j =P_h^{x1}\left(\pi_j =1 \right)$.
We also define a matrix $L$ such that for $i > j$, $L_{ij} = \gamma_{i-j} (g) $ and $L_{ij} = 0$ otherwise. 
Therefore, from \eqref{eq:supvariation} we have
\[
\alpha \leq L\alpha + \beta.
\]
Therefore, 
\begin{equation} \label{eq:variationratio}
\|\alpha\|_{1} \leq \frac{\|\gamma\|_{1}}{1-\|L\|_1} 
= \frac{\sum_{j=1}^\infty
\Var_{j-1} (g) \prod_{k=0}^{j-2}\big(1-\Var_k (g)\big)}{1-\sum_{j=1}^\infty \Var_{j-1} (g)\prod_{k=0}^{j-2}\big(1-\Var_k (g)\big)}.
\end{equation}
Because
\[
\Var_{j-1} (g)\prod_{k=0}^{j-2}\big(1-\Var_k (g)\big)=\prod_{k=0}^{j-2}\big(1-\Var_k (g)\big)-\prod_{k=0}^{j-1}\big(1-\Var_k (g)\big)
\]
we have
\[
\sum_{j=1}^\infty \Var_{j-1} (g)\prod_{k=0}^{j-2}\big(1-\Var_k (g)\big)
= \Var_{{0}}(g) - \prod_{k=0}^{\infty}\big(1-\Var_k (g)\big) \leq 1- \prod_{k=1}^{\infty}\big(1-\Var_k (g)\big).
\]
Therefore, from \eqref{eq:variationratio}, we get
\[
\|\alpha\|_{1} \leq \frac{1-\prod_{k=0}^{\infty}\big(1-\Var_k (g)\big)}{\prod_{k=0}^{\infty}\big(1-\Var_k (g)\big)}.
\]
From Lemma \ref{lemma:comparison}, we have that 
\[
\sum_{j = 1}^n \sup_{y,z\in A^{{\llbracket -\infty, 0 \rrbracket}}}\P^{y,z}(\eta_j \neq \omega_j)  \leq \sum_{j=1}^n P_h^{x1}(\pi_j = 1), 
\]
which concludes the proof. 
\end{proof}

\subsection{Proofs of Theorems \ref{theo:summableconcentration} and  \ref{coro:summableconcentration}}

\begin{proof}[Proof of Theorem \ref{theo:summableconcentration}]
We proceed exactly as in the proof of Theorem
\ref{theo:dobrushinconcentration}, substituting $\Delta (g)$ by $\Gamma (g)$.
\end{proof}

\begin{proof}[Proof of Theorem \ref{coro:summableconcentration}]
Let $\mu$ be a measure compatible with $g$. Let also $\sigma \in A^{\llbracket 0, \infty\rrbracket}$, $a,b \in A$, $n \geq 1$, and $k 
\geq 1$.  Let first define a coupling $\nu^{\sigma, a,b}_k$ between $\mu(\cdot \mid \sigma_0^{k-1}a)$ and $\mu(\cdot \mid \sigma_0^{k-1}
b)$. For all $x, y \in \X^-$, let $\P^{x,y,\sigma, a,b}_k$ be the one-step maximal coupling between $P^{x}(\cdot\mid\sigma^{k-1}_0a)$ 
and $P^{y}(\cdot\mid\sigma^{k-1}_0b)$. We define
\[
\nu^{\sigma, a,b}_k (\cdot)= \int_{\X^-}\int_{\X^-} \P^{x,y,\sigma, a,b}_k(\cdot)
\,\mu\big(\dd x|\sigma^{k-1}_0a\big)\,\mu\big(\dd y|\sigma^{k-1}_0a\big).
\]
From the definition of the above coupling, we have
\begin{align*}
\sup_k\sup_{\sigma}\sup_{a,b}\nu^{\sigma, a,b}_{k}(\eta_{k+j} \neq \omega_{k+j})& \leq \sup_k\sup_{\sigma}\sup_{a,b}\sup_{x,y} 
\P^{x,y,\sigma, a,b}_k(\eta_{k+j} \neq \omega_{k+j})\\
&\leq \sup_{y,z\in A^{{\llbracket -\infty, 0 \rrbracket}}}\P^{y,z}(\eta_j \neq \omega_j)
\end{align*}
where $\P^{y,z}$ is the one-step maximal coupling between $P^{y}$ and $P^{z}$.
Using Proposition \ref{prop:boundingbyvariation} and Theorem \ref{theo:basicGCB}, we conclude that $1+r \leq \Gamma (g)^{-1}$ 
and GCB holds, as we wanted to show. 
\end{proof}


\subsection{Proof of Theorem \ref{theo:PTnoGCB}}

We first need to define two properties: the positive divergence property, and the blowing-up property.

\begin{defi}
We say that an ergodic measure $\mu$ on $A^{\Z}$ satisfies the positive divergence property 
if for any ergodic measure $\nu$ on $A^{\Z}$ different from $\mu$ we have
\[
\liminf_n \frac{1}{n+1}\,\mathbb{E}_{\nu_n}\left[\log \frac{\nu_n}{\mu_n}\right]> 0
\]
where $\mu_n=\mu|_{\mathcal{F}_n}$ and $\nu_n=\nu|_{\mathcal{F}_n}$. 
\end{defi}

We now state two propositions that we will use to prove the next theorem.
\begin{prop}\label{prop:PDP}
Let $g$ be a kernel such that $\inf_{a\in A, x\in \X^-} g(a|x)> 0$ and $\lim_j\Var_j(g) = 0$. 
Suppose that there are two distinct ergodic measures $\mu$ and $\mu$ compatible with $g$. Then the positive divergence property 
does not hold.
\end{prop}
\begin{proof}
Let $\mu$ be an ergodic measure compatible with a kernel $g$, and $\nu$ another ergodic measure compatible with a kernel $h$.
Assume that $\lim_j\Var_j(g) = \lim_j\Var_j(h) = 0$ and $\inf_{a\in A, x\in \X^-} g(a|x)>0$, $\inf_{a\in A, x\in \X^-} h(a|x)>0$. We have 
\begin{align*}
&\frac{1}{n+1}\,\mathbb{E}_{\nu_n}\!\!\left[\log \frac{\nu_n}{\mu_n}\right]=\frac{1}{n+1}\int\log \frac{\nu([x_0^{n}])}{\mu([x_0^n])}\,\nu(\dd 
x_{-\infty}^n)\\
&=\int\frac{1}{n+1}\left(\sum_{j=1}^{n}\log \frac{ \nu\big([x_j]|x_{0}^{j-1}\big)}{\mu\big([x_j]|x_{0}^{j-1}\big)}
+\log\frac{\nu([x_0])}{\mu([x_0])}\right)\nu(\dd x_{-\infty}^n)\\
&=\int\frac{1}{n+1}\sum_{j=1}^{n}\log \frac{ \nu\big([x_0]|x_{-j}^{-1}\big)}{\mu\big([x_0]|x_{-j}^{-1}\big)}\,\nu(\dd x_{-\infty}^0)
+ \frac{1}{n+1}\sum_{x_0\in A}\log\frac{\nu([x_0])}{\mu([x_0])}\,\nu([x_0])
\end{align*}
where the last equality uses shift-invariance of the measures. By uniform continuity of $g$ and $h$ we have 
\[
\log\frac{ \nu\big([x_0]|x_{-j}^{-1}\big)}{\mu\big([x_0]|x_{-j}^{-1}\big)}\xrightarrow[]{j\to\infty} \log\frac{h(x)}{g(x)}
\]
uniformly in $x$, and therefore 
\begin{equation}\label{formuladivergence}
\frac{1}{n+1}\sum_{j=1}^{n}\log \frac{\nu\big([x_0]|x_{-j}^{-1}\big)}{\mu\big([x_0]|x_{-j}^{-1}\big)}
\xrightarrow[]{n\to\infty}\log\frac{h(x)}{g(x)}
\end{equation}
uniformly in $x$ by Ces\`aro lemma.  By the dominated convergence theorem we conclude that 
\[
\lim_n\frac{1}{n+1}\,\mathbb{E}_{\nu_n}\!\!\left[\log \frac{\nu_n}{\mu_n}\right]=\mathbb{E}_\nu\left[\log\frac{h}{g}\right].
\]
Therefore, if $g=h$ and if there are multiple ergodic measures compatible with $g$, then
the measure cannot have the positive divergence property. Indeed, if $\nu$ is an ergodic measure compatible with $g$ but 
different from $\mu$, then the r.h.s. of \eqref{formuladivergence} is equal to $0$, which violates the positive divergence property. 
\end{proof}
For all $n \geq 0$, define the normalised Hamming distance between $\omega$ and $\sigma$ on $A^{n+1}$ by
\begin{equation} \label{eq:hamming}
\bar{d}_n(\sigma,\omega) = \frac{1}{n+1} \sum_{i = 0}^n \mathds{1}_ {\{\sigma_i \neq \omega_i\}}.
\end{equation}

For $F\subset A^{n+1}$ and $\epsilon>0$, $\langle F\rangle_\epsilon$ denotes the $\epsilon$-blowup of $F$, that is
\[
\langle F\rangle_\epsilon
=\{\sigma\in A^{n+1}:\bar{d}_n(\sigma,\omega)\leq \epsilon \;\textup{for some}\;\omega_0^n\in F\}\,. 
\]
\begin{defi}
An ergodic measure $\mu$ has the blowing-up property if given $\epsilon>0$ there is a $\varrho>0$ and $n_0$ such
that if $n\geq n_0$ then $\mu(\langle F\rangle_\epsilon)\geq 1-\epsilon$, for any subset $F\subset A^{n+1}$ for which
$\mu(F)\geq \e^{-(n+1)\varrho}$.
\end{defi}
We make a slight abuse of notation by writing $\mu(F)$ instead of $\mu_n(F)$, or, stated differently, we use the same notation 
for a subset of $A^{n+1}$ and the union of cylinders it generates.

\begin{prop}\label{prop:GCB-BUP}
Suppose that $\mu$ is a probability measure which satisfies GCB with a constant $C$.
For any $n\geq 0$ and any $F\subset A^{n+1}$
such that $\mu(F)>0$, we have
\begin{equation}\label{eq:BUPfromGCB}
\mu(\langle F\rangle_\epsilon)\geq 1-\exp\left(-\frac{n+1}{4C}\left(\epsilon-2\sqrt{\frac{C\log(\mu(F)^{-1})}{n+1}}\,\right)^2\, \right)
\end{equation}
whenever $\epsilon>2\sqrt{\frac{C\log(\mu(F)^{-1})}{n+1}}$. In particular, $\mu$ has the blowing-up property.\newline
\end{prop}
\begin{proof}
Let $n\geq 0$ and $f(\omega_0^n)=\inf_{\sigma_0^n\in F}\sum_{i = 0}^n \mathds{1}_ {\{\sigma_i \neq \omega_i\}}$.
It is obvious that $\delta_i(f)=1$ for $i=0,\ldots,n$. Since $\mu$ satisfies GCB with a constant $C$ by assumption, we get from
\eqref{Chernoff-GCB}
\[
\mu(f>\E_\mu[f]+u)\leq \exp\left(-\frac{u^2}{4C(n+1)} \right), \quad u>0.
\]
We now derive an upper bound for $\E_\mu[f]$. We use \eqref{eq-GCB} with $-\theta f$, where $\theta>0$ will be fixed later on,
to get
\[
\exp(\theta \E_\mu[f]) \, \E_\mu\big[ \exp(-\theta f)\big] \leq \exp\big(C\theta^2 (n+1)\big).
\]
But, by the very definition of $f$, we have
\[
\E_\mu\big[ \exp(-\theta f)\big]\geq \E_\mu\big[ \exp(-\theta f) \mathds{1}_F\big]=\mu(F).
\]
Hence, combining the two previous inequalities, taking the logarithm, and dividing out by $\theta$, we obtain
\[
\E_\mu[f]\leq \inf_{\theta>0} \left\{ C(n+1)\theta + \frac{1}{\theta}\log\big(\mu(F)^{-1}\big)\right\}
\]
which gives
\[
\E_\mu[f]\leq 2\sqrt{C(n+1) \log(\mu(F)^{-1})}\,.
\]
To finish the proof of \eqref{eq:BUPfromGCB}, observe that $\mu(f>\epsilon)=\mu(\langle F\rangle_\epsilon^c)$.\newline
Now, if we fix $\epsilon>0$ and take $F$ such that $\mu(F)\geq \exp(-(n+1)\varrho)$, for some $\varrho>0$ to be chosen later on,
subject to the condition $\epsilon>2\sqrt{C\varrho}$, we get from \eqref{eq:BUPfromGCB} that, for all $n\geq 0$,
\[
\mu(\langle F\rangle_\epsilon)\geq 1-\exp\left(-\frac{n+1}{4C}\left(\epsilon-2\sqrt{C\varrho}\,\right)^2\, \right).
\]
We now take $\varrho=\epsilon^2/(4C)$ which gives
\[
\mu(\langle F\rangle_\epsilon)\geq 1-\epsilon
\]
for all $n\geq n_0:=\lfloor 4\epsilon^{-2}\log(\epsilon^{-1})\rfloor$.
We thus proved that GCB implies the blowing-up property.
\end{proof}
We are ready to prove the following result, which is of independent interest. 

\begin{proof}[Proof of Theorem  \ref{theo:PTnoGCB}]
If $\inf_{a\in A, x\in \X^-} g(a|x)> 0$, then the alphabet has to be finite (see Remark \ref{rem:inf}). 
It is proved in \cite{marton/shields/1994}  that, for finite alphabet ergodic stationary processes,  the blowing-up 
property implies the positive divergence property.
But by Proposition \ref{prop:PDP}, we cannot have the latter property since we assume that there are at least two
ergodic measures compatible with the kernel. Hence the blowing-up properties does not hold. But then, by Proposition
\ref{prop:GCB-BUP}, we cannot have GCB.
\end{proof}

\subsection{Proof of Theorem \ref{theo:renewal}}

Recall the definitions of $\tilde g$, the kernel of the renewal measure ${\tilde\mu}$, and the distribution $f_n,n\ge1$ of the distance between consecutive 
$1$'s. In order to prove Theorem \ref{theo:renewal} we will use a well-known relation between the renewal process and an $\mathbb{N}$-valued  Markov 
chain. Indeed, let $F:\N^\N\rightarrow \{0,1\}^\N$ be the deterministic coordinate-wise function defined by $\big(F(\sigma)\big)_i=\mathds1_{[0]}
(\sigma_i),i\in\N$. We refer the reader to \cite{meyn/tweedie/2012}, where in particular it is explained that $\tilde\mu=\nu\circ F^{-1}$ 
where $\nu$ is the Markov measure  with transition matrix $Q$ given by
\[
Q(m,0)=1-Q(m,m+1)=\frac{f_{m+1}}{\sum_{i\ge m+1}f_i}\,,\,\,\quad m\ge0.
\] 

\begin{proof}[Proof of Theorem \ref{theo:renewal}]
We start by proving sufficiency. 
Suppose first that $\sum_nf_nr^n<\infty$ for some $r>1$. Then the time $\tau^0$ separating two consecutive $0$'s for the Markov 
measure $\nu$ has distribution $f_n, n\ge1$, by construction. Therefore, $\E_{\nu^0}[r^{\tau^0}]<\infty$ for the same $r$, where
$\E_{\nu^0}$ denotes the expectation with respect to the measure of the Markov chain initiated at state $0$.   Following 
\cite{meyn/tweedie/2012}, this characterizes $\nu$ as a geometrically ergodic Markov measure (in fact, it is equivalent, see 
\cite[Section 15.1.4]{meyn/tweedie/2012}). Using the result of \cite{dedecker/gouezel/2015}, we conclude that $\nu$ satisfies GCB, and,
as a coordinate-wise image of $\nu$, the renewal process $\tilde\mu$ also has GCB. This last step is a consequence of
\cite[Theorem 7.1]{kontorovich2008concentration}.

We now prove necessity. Suppose that $\tilde\mu$ satisfies   GCB. Then, for some sufficiently small $c>0$,
\begin{align*}
\tilde\mu([0^{n+1}])&=\tilde\mu\left(\left\{\omega:\frac{1}{n+1}\sum_{i=0}^n\omega_i=0\right\}\right)\\
&\le \tilde\mu\left(\left\{\omega:\frac{1}{n+1}\sum_{i=0}^n\omega_i-\tilde\mu([1])\le -\frac{\tilde\mu([1])}{2}\right\}\right)\\&\le \e^{-cn}.
\end{align*}
On the other hand, by shift-invariance we have
\[
\tilde\mu([0^{n+1}])=\sum_{i\ge n+1}\tilde\mu([10^i])=\tilde\mu([1])\sum_{i\ge n+1}\sum_{j\ge i}f_j\ge\tilde\mu([1])\sum_{i\ge n+1}f_i.
\] 
This means that $\tilde\mu([1])\sum_{i\ge n+1}f_i\le \e^{-cn}$ which implies that $\sum_nf_nr^n<\infty$ for some $r>1$. 
\end{proof}

\subsection{Proof of Theorem  \ref{concentrationformoregeneralfunctions}}

Take an arbitrary $\eta\in A^{\N}$ and for $n\geq 0$ define
$f_n(\omega):=f(\omega_0^n\eta_{n+1}^\infty)$. By construction we have $\|f-f_n\|_\infty\leq \mathrm{var}_n(f)\to 0$.
Now, for each $i$, $\delta_i(f-f_n)$ goes to $0$ when $n\to\infty$ since for all $n\geq i$ it is easy to check that
\[
\delta_i(f-f_n)\leq 2 \,\mathrm{var}_n(f).
\]
We have the inequality 
\[
(\delta_i(f-f_n))^2\leq 4 \delta_i(f)^2, \;\forall i,n.
\]
Therefore, since $\|\ushort{\delta}(f)\|_2<\infty$ by assumption, we can use the dominated convergence theorem (for the counting 
measure on the set of nonnegative integers) to get
\[
\|\ushort{\delta}(f-f_n)\|_2\to 0.
\]
Finally, using GCB for $f_n$, and the obvious fact that $\delta_i(f+g)\leq \delta_i(f) + \delta_i(g)$, we get
\begin{align*}
\E_\mu\!\left[\e^{f  -\E_{\mu}[f]}\right]
& \leq \E_\mu\!\left[\e^{f_n  -\E_{\mu}[f_n]}\right] \e^{2\|f-f_n\|_\infty}\\
& \leq \e^{C \|\ushort{\delta} (f_n)\|^2_2} \e^{2\|f-f_n\|_\infty} \\
& \leq \e^{C \|\ushort{\delta}(f)\|^2_2}  \e^{2C \|\ushort{\delta} (f)\|_2\|\ushort{\delta} (f-f_n)\|_2}
\e^{C \|\ushort{\delta} (f-f_n)\|_2^2}\e^{2\|f-f_n\|_\infty}
\end{align*}
where the third inequality follows by writting $f_n=f_n-f+f$ and expanding $(\delta_i(f_n-f+f))^2$. 
The result follows by letting $n$ tend to infinity.

\subsection{Proof of Theorem  \ref{theo:DKW}}

Define ${f}= \|\hat{\rho}_{n,k} -\rho\|_\infty$. Recall that $\hat{\rho}_{n,k}(\sigma)=\hat{\rho}_{n,k}(\sigma,\cdot)$.
For all $n \geq 1$, we have $\|\ushort{\delta}(f)\|_2 = 1$, hence, from Theorem \ref{theo:dobrushinconcentration}, we have
\begin{equation}\label{vendredi}
\mu(\|\hat{\rho}_{n,k} -\rho\|_\infty -\E_{\mu}[\|\hat{\rho}_{n,k} -\rho\|_\infty] > u)\leq \exp \big( -2 (n-k+2)\, \Gamma (g)^{2}u^2\,\big).
\end{equation}
Therefore, to prove Theorem  \ref{theo:DKW}, we only need to find a good upper bound for $\E_{\mu}[\|\hat{\rho}_{n,k} -\rho\|_\infty]$.
Here, we follow the argument used in \cite{kontorovich2014uniform}.
By Jensen's inequality, and since $\E_{\mu}[\hat{\rho}_{n,k}(\sigma)]=\rho(\sigma)$, we have
\begin{align}
\big(\E_{\mu}[\|\hat{\rho}_{n,k} -\rho\|_\infty]\big)^2
&\leq \E_{\mu}\big[\|\hat{\rho}_{n,k} -\rho\|_\infty^2\big] \notag 
\leq \E_{\mu}\!\left[\,\sum_{\sigma \in A^{\llbracket 1,k\rrbracket}}(\hat{\rho}_{n,k}(\sigma) -\rho(\sigma))^2\right] \notag \\
&\leq \sum_{\sigma \in A^{\llbracket 1,k\rrbracket}}  \big(\E_{\mu}\big[\hat{\rho}_{n,k}(\sigma)^2\big] -\rho(\sigma)^2\big).
\label{eq:DKW1}
\end{align}
Recall that, for all $S\subset\Z$ and $\sigma\in A^S$, we define the projection function associated to all indices $i, j\in S$,  
$j\leq i$,  by $\pi_j^i(\sigma)=\sigma_j^i $. For all $\sigma \in A^{\llbracket 1,k\rrbracket}$, we have
\begin{align*}
&\E_{\mu}\!\left[\hat{\rho}_{n,k}(\sigma)^2\right]\\
&= \frac{1}{(n-k+2)^2}\,\E_{\mu}\!\left[\left(\sum_{i=0}^{n-k+1}\mathds{1}_{\sigma}\circ\pi_i^{i+k-1}\right)^2\,\right]\\
& =\frac{1}{(n-k+2)^2}\,\E_{\mu}\!\left[\sum_{i=0}^{n-k+1} \mathds{1}_{\sigma}\circ\pi_i^{i+k-1}
+ 2\sum_{j=1}^{n-k+1}\sum_{i=0}^{j-1} \big(\mathds{1}_{\sigma}\circ\pi_i^{i+k-1}\big)\big(\mathds{1}_{\sigma}\circ\pi_j^{j+k-1}\big)\right]\\
& = \frac{\rho(\sigma)}{n-k+2} + \frac{2}{(n-k+2)^2}\sum_{j=1}^{n-k+1}\sum_{i=0}^{j-1}\mu\big(\pi_i^{i+k-1} = \sigma, \pi_j^{j+k-1} = \sigma\big)\\
& = \frac{\rho(\sigma)}{n-k+2} +
\frac{2}{(n-k+2)^2}\sum_{j=1}^{n-k+1}\sum_{i=0}^{j-1}\rho(\sigma)\mu\big(\pi_j^{j+k-1} = \sigma \big| \pi_i^{i+k-1} = \sigma\big)\\
& \leq \frac{\rho(\sigma)}{n-k+2} + \frac{2}{(n-k+2)^2}\sum_{j=1}^{n-k+1}\sum_{i=0}^{j-1}\rho(\sigma)\big(\rho(\sigma)
+ \big|\mu\big(\pi_j^{j+k-1} = \sigma \big| \pi_i^{i+k-1} = \sigma\big) - \rho(\sigma)\big|\big).
\end{align*}
Now let $j^* = \max\{j,i+k\}$. For all $a \in A$, we have
\begin{align*}
& \big|\mu\big(\pi_j^{j+k-1} = \sigma \big| \pi_i^{i+k-1} = \sigma\big) - \rho(\sigma)\big| \\
& \leq \sup_{\tilde{\sigma} \in A^{\llbracket 1,k\rrbracket}}|\mu(\pi_{j^*}^{j+k-1} = \pi_{j^*}^{j+k-1}(\sigma) | \pi_i^{i+k-1} = \sigma) - \mu(\pi_{j^*}^{j+k-1} = \pi_{j^*}^{j+k-1}(\sigma) | \pi_i^{i+k-1} = \tilde{\sigma})|\\
& = \sup_{\tilde{\sigma} \in A^{\llbracket 1,k\rrbracket}}
\big|\mu\big(\pi_{j^*-i-k}^{j-i-1} = \pi_{j^*-i-k}^{j-i-1}(\sigma) \big| \pi_{-k}^{-1} = \sigma\big) -
\mu\big(\pi_{j^*-i-k}^{j-i-1} = \pi_{j^*-i-k}^{j-i-1}(\sigma) \big| \pi_{-k}^{-1} = \tilde{\sigma}\big)\big|\\
& \leq \sup_{x,y \in \X^-}
\big|P^{x}\big(\pi_{j^*-i-k}^{j-i-1} = \pi_{j^*-i-k}^{j-i-1}(\sigma) \big) - P^{y}\big(\pi_{j^*-i-k}^{j-i-1} = \pi_{j^*-i-k}^{j-i-1}(\sigma) \big)\big|\\
& \leq \sup_{x,y \in \X^-}\P^{x,y}\big(\eta_{j^*-i-k}^{j-i-1} \neq \omega_{j^*-i-k}^{j-i-1}\big)\\
& \leq \sup_{x,y \in \X^-} \sum_{\ell = j^*-i-k}^{j-i-1} \P^{x,y}\big(\eta_\ell \neq \omega_\ell\big)
\end{align*}
where $\P^{x,y}$ is the one-step maximal coupling between $P^x$ and $P^y$.
Observe that $j^*-i-k=\max\{j-i-k,0\}$.
Coming back to the estimation of $\E_{\mu}\left[\hat{\rho}_{n,k}(\sigma)^2\right]$, we have
\begin{align*}
& \E_{\mu}\!\left[\hat{\rho}_{n,k}(\sigma)^2\right]\\
& \leq \frac{\rho(\sigma)}{n-k+2} + \frac{2}{(n-k+2)^2}\sum_{j=1}^{n-k+1}\,\sum_{i=0}^{j-1}\rho(\sigma)\bigg(\rho(\sigma) +\!\! 
\sup_{x,y \in \X^-} \sum_{\ell = \max\{j-i-k,0\}}^{j-i-1} \P^{x,y}(\eta_\ell \neq \omega_\ell)\bigg) \\
& \leq \frac{\rho(\sigma)}{n-k+2} + \rho(\sigma)^2+ \frac{2\rho(\sigma)(n-k+1)}{(n-k+2)^2}\sum_{i=0}^{n-k}\,\sum_{\ell=\max\{n-2k-i+1,0\}}^{n-k-i}
\,\sup_{x,y \in \X^-}  \P^{x,y}(\eta_\ell \neq \omega_\ell)\big) \\
& \leq \frac{\rho(\sigma)}{n-k+2} + \rho(\sigma)^2 + \frac{2\rho(\sigma)\,k}{n-k+2}\;\sum_{i=0}^{n-k} \sup_{x,y \in \X^-}\!\!\P^{x,y}
(\eta_{i} \neq\omega_{i})\\
& \leq \frac{\rho(\sigma)}{n-k+2} + \rho(\sigma)^2 + \frac{2\rho(\sigma)\,k}{n-k+2}\;\sum_{i=0}^{\infty} \sup_{x,y \in \X^-}\!\!\P^{x,y}
(\eta_{i} \neq \omega_{i}) \\
& \leq \rho(\sigma)^2 + \frac{2\rho(\sigma)\, k}{(n-k+2)\Gamma (g)}
\end{align*}
where we used Proposition \ref{prop:boundingbyvariation} in the last inequality.
Finally, we obtain from \eqref{eq:DKW1} that
\[
\E_{\mu}[\|\hat{\rho}_{n,k} -\rho\|_\infty] \leq \sqrt{\frac{2k}{(n-k+2)\Gamma (g)}}
\]
which is the desired bound. Combining this bound with \eqref{vendredi} and rescaling $u$ in an obvious way, we finally obtain 
\eqref{ineq:DKW}.

\subsection{Proof of Theorems \ref{theo:dbar}}

Recall the definition \eqref{eq:hamming} of the Hamming distance between $\omega,\sigma\in A^{\llbracket 0, n\rrbracket}$.
The $\bar{d}$-distance between two probability measures $\mu_n, \nu_n$ on $A^{\llbracket 0, n\rrbracket}$ is
\[
\bar{d}(\mu_n, \nu_n) = \inf \sum_{\sigma, \tilde{\sigma} \in A^{n+1}} \bar{d}_n(\sigma, \tilde{\sigma})\, 
\P_n(\sigma, \tilde{\sigma})
\]
where the infimum is taken over all couplings $\P_n$ of $\mu_n$ and $\nu_n$. 

Consider functions $f: A^{n+1} \to \R$ such that, for $j \in \llbracket 0,n \rrbracket$, $\delta_j(f) \leq 1/(n+1)$. Such functions are 
$1$-Lipschitz with respect to the Hamming distance because for all $\sigma, \eta \in A^{n+1}$
\[
|f(\sigma)-f(\eta)|
\leq \sum_{j=0}^n \delta_j(f) \, \mathds{1}_{\{\sigma_j\neq \eta_j\}}
\leq \frac{1}{n+1}\sum_{j=0}^n \mathds{1}_{\{\sigma_j\neq \eta_j\}}=\bar{d}_n(\sigma,\eta).
\]
Let $g$ be a kernel and $\mu$ a compatible measure satisfying the conditions of the theorem. Then, Theorems 
\ref{coro:dobrushinconcentration} and \ref{coro:summableconcentration} state that, for such functions,
for all $\theta \in\R$ and $n \geq 0$, 
\[
\E_\mu\left[ \e^{\theta(f-\E_\mu[f])}\right] \leq \exp\left( \frac{\mathcal{C}^{-2} (n+1)^{-1}\theta^2}{8}\right). 
\]
The main observation is that this is equivalent, according to \cite[Theorem 3.1]{bobkov1999exponential}, to having
\[
\bar{d}_n(\nu_n, \mu_n) \leq \frac{1}{\mathcal{C}} \sqrt{\frac{1}{2(n+1)}\,\mathbb{E}_{\nu_n}\left[\log \frac{\nu_n}{\mu_n}\right]}
\]
where $\mu_n=\mu|_{\mathcal{F}_n}$ and $\nu_n$ is any  probability measure on $A^{n+1}$.
{Consider now the measure $\nu$ compatible with $h$ as given in the statement of the theorem and let $\nu_n$ be $\nu|_{\mathcal{F}_n}$}. We have by stationarity
\cite{shields/1996} that 
\[
\bar d(\mu,\nu)=\lim_n\bar{d}_n(\mu_n, \nu_n) .
\]
So the proof of the proposition is concluded since we have that (recall the proof of Theorem \ref{theo:PTnoGCB} {above})
\[
\lim_n\frac{1}{n+1}\,\mathbb{E}_{\nu_n}\left[\log \frac{\nu_n}{\mu_n}\right]=\mathbb{E}_{\nu}\left[\log \frac{h}{g}\right].
\] 

\bigskip

\noindent {\bf Acknowledgements.}
SG thanks CNRS, FAPESP  (19805/2014 and 2017/07084-6) as well as CNPq Universal (439422/2018-3)  for financial support. 
DYT thanks \'Ecole Polytechnique for a 2-month fellowship. SG and DYT thank the CPHT for its hospitality during several stays.  

\bibliographystyle{abbrv}
\bibliography{Concentration_SCUM}

\end{document}